\documentclass[12pt,reqno]{amsart}
\usepackage{cite}

\usepackage{verbatim,amscd,amssymb}
\usepackage{graphicx}
\usepackage{amsmath}

\usepackage{inputenc}

\usepackage{enumitem}
\usepackage{float}

\usepackage[active]{srcltx}

\graphicspath{ {./images/} }

\usepackage{fontenc}

 \usepackage{float}

\newtheorem{theorem}{Theorem}[section]
\newtheorem{lemma}[theorem]{Lemma}
\newtheorem{cor}[theorem]{Corollary}

\theoremstyle{definition}
\newtheorem{definition}[theorem]{Definition}

\theoremstyle{remark}

\numberwithin{equation}{section}

\newcommand{\R}{\mathbb R}
\newcommand{\uc}{\mathbb{S}}
\newcommand{\al}{\alpha}
\newcommand{\be}{\beta}

\newcommand{\ol}{\overline}

\newcommand{\sha}{\succ\mkern-14mu_s\;}
\newcommand{\si}{\sigma}

\begin{document}

\date{August 20, 2019}

\title[Over-rotation intervals of bimodal interval maps]{Over-rotation intervals of bimodal interval maps}

\author{Sourav Bhattacharya}

\author{Alexander Blokh}

\address[Sourav Bhattacharaya and Alexander Blokh]
{Department of Mathematics\\ University of Alabama at Birmingham\\
Birmingham, AL 35294}



\begin{abstract}
We describe all bimodal over-twist patterns and
give an algorithm allowing one to determine what the left endpoint of
the over-rotation interval of a given bimodal map is. We then define a
new class of polymodal interval maps called \emph{well behaved}, and generalize onto them the above results.
\end{abstract}

\maketitle

\section{INTRODUCTION}

The paper is devoted to studying rotation theory for interval maps. We
divide Introduction in several subsections.

\subsection{Motivation for rotation theory}
The celebrated Sharkovsky Theorem \cite{S, shatr} elucidates the rich
combinatorial behavior of periodic orbits of continuous interval maps.
To state it let us first introduce the \emph{Sharkovsky ordering} for
the set $\mathbb{N}$ of positive integers:
$$3\sha 5\sha 7\sha\dots\sha 2\cdot3\sha 2\cdot5\sha 2\cdot7 \sha \dots $$
$$
\sha\dots 2^2\cdot3\sha 2^2\cdot5\sha 2^2\cdot7\sha\dots\sha 8\sha
4\sha 2\sha 1$$ If $m\sha n$, say that $m$ is {\it\textbf{shar}per} than
$n$. Denote by $Sh(k)$ the set of all positive integers $m$ such that
$k\sha m$, together with $k$, and by $Sh(2^\infty)$ the set
$\{1,2,4,8,\dots\}$ which includes all powers of $2$. Denote also by
$P(f)$ the set of the periods of cycles of a map $f$ (by the
\emph{period} we mean the \emph{least} period). Theorem \ref{t:shar}
was proven by A. N. Sharkovsky.

\begin{theorem}[\cite{S, shatr}]\label{t:shar}
If $f:[0,1]\to [0,1]$ is a continuous map, $m\sha n$ and $m\in
P(f)$, then $n\in P(f)$. Therefore there exists $k \in \mathbb{N}
\cup \{2^\infty\}$ such that $P(f)=Sh(k)$. Conversely, if $k\in
\mathbb{N}\cup \{2^\infty\}$ then there exists a continuous map
$f:[0,1]\to [0,1]$ such that $P(f)=Sh(k)$. 		
\end{theorem}

The above theorem provides a full description of possible sets of
periods of cycles of continuous interval maps. Moreover, it shows that
various periods \emph{force} one another in the sense that if $m \sha
n$ then, for a continuous interval map $f$, the existence of a cycle of
period $m$ \emph{forces} the existence of a cycle of period $n$.

However, the period is a rough characteristic of a cycle as there are a
lot of cycles of the same period. A much finer way of describing cycles
is by considering the permutations induced by cycles. Moreover, one
can define \emph{forcing} relation among permutations in a natural way,
and, by \cite{Ba}, this relation is a partial order. Still, a drawback
here is that while the forcing relation among cyclic permutations
is very fine, the results are much more complicated than for periods
and a transparent description of possible sets of permutations
exhibited by the cycles of a continuous interval map is apparently
impossible (see, e.g., \cite{alm00}). This motivates one to look for
another, middle-of-the-road way of describing cycles, a way not as crude as
periods but not as fine as permutations, which would still allow for a
transparent description.

\subsection{Functional rotation numbers and rotation theory on the circle}\label{ss:fun-rot}
Looking for a new way of describing cycles of interval maps, it is
natural to use as prototypes the results for circle maps of degree one
due to Misiurewicz \cite{mis82} who used the notion of the {\it
rotation number}. This notion was first introduced by Poincar\`e
\cite{poi} for circle homeomorphisms, then extended to circle maps of
degree one by Newhouse, Palis and Takens \cite{npt83}, and then
studied, e.g., in \cite{bgmy80}, \cite{ito81}, \cite{cgt84},
\cite{mis82}, \cite{mis89}, \cite{almm88} (see \cite{alm00} with an
extensive list of references). One can define rotation numbers in a
variety of cases using the following approach \cite{mz89},
\cite{zie95}. Let $X$ be a compact metric space, $\phi:X\to\mathbb R$ be a
bounded measurable function (often called an \emph{observable}),
$f:X\to X$ be a continuous map. Then for any $x$ the set
$I_{f,\phi}(x)$ of all limits of the sequence
${\frac1n}\sum^{n-1}_{i=0}\phi(f^ix)$ is called the {\it
$\phi$-rotation set} of $x$. It is easy to see that $I_{f,\phi}(x)$ is a closed
interval. If $I_{f,\phi}(x)=\{\rho_\phi(x)\}$ is a singleton, then the
number $\rho_\phi(x)$ is called the {\it $\phi$-rotation number} of
$x$. The union of all $\phi$-rotation sets of points of $X$ is called
the \emph{$\phi$-rotation set} of the map $f$ and is denoted by
$I_f(\phi)$.

If $x$ is an $f$-periodic point of period $n$ then its rotation number
$\rho_\phi(x)$ is well-defined, and a useful related concept of the
\emph{$\phi$-rotation pair} of $x$ can be introduced; namely, the pair
$({\frac1n}\sum^{n-1}_{i=0}\phi(f^ix), n)$ is said to be the
\emph{$\phi$-rotation pair} of $x$. Evidently, for all points from the
same cycle their $\phi$-rotation pairs $(t, n)$ are the same, and their
$\phi$-rotation numbers are $\frac{t}{n}$.

For functions $\phi$ related to the dynamics of $f$ one might get
additional results about $\phi$-rotation sets; e.g., this happens for
rotation numbers in the circle degree one case \cite{mis82}. Let
$f:S^1\to S^1$ be a continuous map of degree $1$, $\pi:\mathbb R\to S^1$
be the natural projection which maps an interval $[0,1)$ onto the whole
circle. Fix a lifting $F$ of $f$. Define $\phi_f:S^1\to \mathbb R$ so that
$\phi_f(x)=F(X)-X$ for any point $X\in \pi^{-1}x$; then $\phi_f$ is
well-defined, the classical rotation set of a point $z$ is
$I_{f,\phi_f}(z)=I_f(z)$ and the classical rotation number of $z$ is
$\rho_{f,\phi_f}(z)=\rho(z)$ whenever exists.

The {\it rotation set} of the map $f$ is $I_f=\cup I_f(x)$; it follows
from \cite{npt83},\cite{ito81} that $I_f$ is a closed interval (cf
\cite{B1}). The sum $\sum^{n-1}_{i=0}\phi_f(f^ix)=m$ taken along the
orbit of an $n$-periodic point $x$ is an integer which defines a pair
$(m,n)\equiv rp(x)$ called the {\it rotation pair} of $x$; denote the
set of all rotation pairs of periodic points of $f$ by $RP(f)$. For
real $a\le b$ let $N(a,b)=\{(p,q)\in \mathbb Z^2_+: p/q\in (a,b)\}$ (in
particular $N(a,a)=\emptyset$). For $a\in \mathbb R$ and $l\in \mathbb
Z_+\cup \{2^\infty\}$ let $Q(a,l)$ be empty if $a$ is irrational;
otherwise let it be $\{(ks,ns): s\in Sh(l)\}$ where $a=k/n$ with $k,n$
coprime.

\begin{theorem}[\cite{mis82}]\label{t:misiu} For a continuous circle map $f$ of degree
$1$ such that $I_f=[a,b]$ there exist $l,r\in \mathbb Z_+\cup
\{2^\infty\}$ such that $RP(f)=N(a,b)\cup Q(a,l)\cup Q(b,r)$.
\end{theorem}

Observe that, equivalently, one can talk about continuous degree one
maps $F$ of the real line to itself. Each such map $F$ is a lifting of (is locally
monotonically semiconjugate
to) a continuous degree one map $f$ of the unit circle. For brevity in
the future we will talk about \emph{(lifted) periodic points of $F$} meaning
points $x\in \R$ that project to $f$-periodic points of the unit circle
$\uc$ under a canonical projection $\pi:\R\to \uc$. Then
Theorem~\ref{t:misiu} can be viewed as a result describing possible
rotation pairs and numbers of periodic points of continuous degree one
maps of the real line to itself.

\subsection{Rotation theory on the interval}
The choice of $\phi_f$ is crucial for Theorem~\ref{t:misiu} and is
dynamically motivated. It turns out that, with an appropriate choice of
the observable, results similar to Theorem~\ref{t:misiu} can be obtained
for interval maps too. First it was done when the rotation numbers for
interval maps were introduced in 
[Blo94 - Blo95b] 
(see also \cite{bk98}).

Namely, one defines the \textit{rotation pair} of a non-degenerate
cycle as $(p,q)$, where $q$ is the period of the cycle and $p$ is the
number of its
elements which are mapped to the left of themselves. 
Let us denote the rotation pair of a cycle $P$ by $rp(P)$ and the set
of the rotation pairs of all cycles of a map $f$ by $RP(f)$. The number
$p/q$ is called the \textit{ rotation number} of the cycle $P$. We
introduce the following partial ordering among all pairs of integers
$(p,q)$ with $0<p<q$. We will write $(p,q)\gtrdot (r,s)$ if either
$1/2\le r/s<p/q$, or $p/q<r/s\le 1/2$, or $p/q=r/s=m/n$ with $m$ and
$n$ coprime and $p/m\sha r/m$ (notice that $p/m,r/m \in \mathbb{N}$).

\begin{theorem}\label{t:b11}
If  $f:[0,1]\to [0,1]$ is continuous, $(p,q)\gtrdot (r,s)$ and
$(p,q)\in RP(f)$ then $(r,s)\in RP(f)$. 	
\end{theorem}

This theorem makes it possible to give a full description of the sets
of rotation pairs for continuous interval maps (as in Theorem 1.1, all
theoretically possible sets really occur), see \cite{B2}. This
description is similar to the one for circle maps of degree one (see
\cite{mis82}).

A further development came with another choice of the observable made
in \cite{BM1}. The results of \cite{BM1} imply those of [Blo94 -
Blo95b], hence from now on we will study the new invariants introduced in
\cite{BM1}. Let $f:[0,1]\to [0,1]$ be continuous, $Per(f)$ be its set
of periodic points, and $Fix(f)$ be its set of fixed points. It is easy
to see that if $Per(f)=Fix(f)$ then $\omega(y)$ is a fixed point for
any $y$. Assume from now on that $Per(f)\neq Fix(f)$ and define a
function $\chi_f=\chi$ as follows:

$$\chi(x)=\begin{cases} 1/2&\text{if $(f(x)-x)(f^2(x)-f(x))\le 0$,}\\ {0}&\text{if
$(f(x)-x)(f^2(x)-f(x))>0$.}\end{cases}$$


For any non-fixed periodic point $y$ of period $p(y)$ the integer
$l(y)=$ $\sum^{n-1}_{i=0}\chi(f^iy)$ is at most $p(y)/2$ and is the
same for all points from the orbit of $y$. The pair $orp(y)=(l(y),
p(y))$ is called the {\it over-rotation pair} of $y$, and {\it coprime
over-rotation pair} if $p,q$ are coprime. Notice that in an
over-rotation pair $(p,q)$ both $p$ and $q$ are integers and $0<p/q
\leq 1/2$.

The set of all over-rotation pairs of periodic non-fixed points of $f$
is denoted by $ORP(f)$ and the $\chi$-rotation number
$\rho_\chi(y)=\rho(y)=l(y)/p(y)$ is called the {\it over-rotation
number} of $y$. Observe that by Theorem~\ref{t:shar} and by the
assumption that $Per(f)\neq Fix(f)$ it follows that $f$ has a point of
period $2$ and that the over-rotation number of this point is $1/2$; in
other words, the set of all over-rotation numbers of periodic points of
$f$ includes $1/2$ and, therefore, $1/2$ belongs to the union of all
$\chi$-rotation sets $I_{f, \chi}(x)$ defined above.


\begin{theorem}[\cite{BM1}]\label{t:bm2} If $(p,q) \gtrdot (k,l)$ and
$(p,q)\in ORP(f)$ then $(k,l)\in ORP(f)$.
\end{theorem}

Theorem~\ref{t:bm2} implies Theorem~\ref{t:shar}. Indeed, let $f$ be an
interval map and consider odd periods. For any $2n+1$ the closest to
$1/2$ over-rotation number of a periodic point of period $2n+1$ is
$\frac{n}{2n+1}$. Clearly $\frac{n}{2n+1}<\frac{n+1}{2n+3}<\frac12$.
Hence for any periodic point $x$ of period $2n+1$ its over-rotation
pair
$orp(x)$ is $\gtrdot$-stronger than the pair $(n+1,2n+3)$, 
and by Theorem~\ref{t:bm2} the map $f$ has a point of period $2n+3$.
Also, for any $m$ we have $(n,2n+1)\gtrdot (m,2m)$, so by
Theorem~\ref{t:bm2} the map $f$ has a point of any even period $2m$.
Applying this to the maps $f,f^2,f^4,\dots$ one can prove
Theorem~\ref{t:shar} for all periods but the powers of $2$; additional
arguments covering the case of powers of $2$ are quite easy.

Theorem~\ref{t:bm2} implies a full description of sets $ORP(f)$ for
interval maps, close to that from Theorem~\ref{t:misiu}; in fact, the same description
applies to sets $RP(f)$ (see
Theorem~\ref{t:b11}) except that a set $ORP(f)$ is always located to the left of $1/2$
because over-rotation numbers are less than or equal to $\frac12$). To state the corresponding
result we introduce new notation. Let
$\mathcal{M}$ be the set consisting of 0, $1/2$, all irrational numbers
between 0 and $1/2$, and all pairs $(\alpha,n)$, where $\alpha$ is a
rational number from $(0,1/2]$ and $n\in\mathbb{N} \cup \{2^\infty\}$.
Then for $\eta\in\mathcal{M} $ the set $Ovr(\eta)$ is equal to the
following. If $\eta$ is an irrational number, 0, or $1/2$, then
$Ovr(\eta)$ is the set of all over-rotation pairs $(p,q)$ with $ \eta
<p/q \leq 1/2$. If $\eta=(r/s,n)$ with $r, s$ coprime, then $Ovr(\eta)$ is
the union of the set of all over-rotation pairs $(p,q)$ with $r/s <p/q
\leq 1/2$ and the set of all over-rotation pairs $(mr,ms)$ with $m \in
Sh(n)$. Notice that in the latter case, if $n\neq 2^\infty$, then $
Ovr(\eta)$ is equal to the set of all over-rotation pairs $(p,q)$ with
$(nr,ns)\gtrdot (p,q)$, plus $(nr,ns)$ itself. 

\begin{theorem}\label{t:ovr}
	Given a continuous interval map $f$,
	there exists $ \eta \in \mathbb{N} $ such that $ORP(f)=Ovr(\eta)$.
	Conversely, if $\eta\in\mathcal{M} $ then there exists a continuous map $f:[0,1]\rightarrow
	[0,1]$such that $ORP(f)=Ovr(\eta)$.
\end{theorem}

The closure of the set of over-rotation numbers of periodic points of
$f$ is an interval $I_f=[\rho_f, 1/2], 0\le \rho_f\le 1/2,$ called the
{\it over-rotation interval} of $f$. 

\subsection{The role of periodic points}
As the reader may have noticed by now, we introduced the functional
rotation numbers considering all points $x$; on the other hand, our
main focus has been on periodic points and their over-rotation numbers and
pairs. It is then natural to consider the role of functional rotation
numbers of periodic points in functional rotation sets of maps,
including their density in those sets. On the other hand, evidently,
Theorem~\ref{t:ovr} is modeled after Theorem~\ref{t:misiu}. It is then
also natural to consider other parallels between classical rotation
numbers defined for the circle maps of degree one, and over-rotation
numbers for interval maps. In fact, this is one of the main ideas of
the present paper as applies to $N$-bimodal and other similar classes
of interval maps. However first we study a different (but related)
analogy between circle maps of degree one and interval maps, namely we
study the role of periodic points and their (over-)rotation numbers in
both circle and interval cases.

Indeed, even without the complete description of possible rotation
pairs of periodic points of circle maps of degree one given in
Theorem~\ref{t:misiu}, one can show that for any degree one continuous
circle map $f$ either $f$ is monotonically conjugate to an irrational
rotation of the circle, or the rotation numbers of its periodic points
are dense in its rotation interval. This fact is related to a more
general problem of establishing the connection between the $\phi$-rotation
numbers of {\it periodic points} of a map $f$, and the
\emph{$\phi$-rotation set} $I_f(\phi)$ of the map $f$ for any function
$\phi$. We describe this connection in the case of
interval maps and circle maps of degree one (see, e.g., \cite{B3}); our
explanation is based upon the so-called ``spectral decomposition'' for
one-dimensional maps \cite{blok86, blok87a, blok87b}, \cite{blo95a} (we
will only use it for interval and circle maps).

To state the appropriate results we need a few basic definitions as
well as a couple of less standard ones (see, e.g., \cite{dgs76}). Given
a cycle $A$ of period $n$, a unique invariant probability measure
$\nu_A$ concentrated on $A$ is the measure assigning to each point of
$A$ the weight $\frac1n$; let us call $\nu_f$ a \emph{CO-measure}
\cite{dgs76} (comes from ``\textbf{c}losed \textbf{o}rbit''). Recall that for Borel
measures on compact spaces one normally considers their \emph{weak}
topology defined by the continuous functions \cite{dgs76}.

\begin{theorem}[\cite{blok86, blok87a, blok87b, blo95a}]\label{t:codense} Suppose that
$f:I\to I$ is a continuous interval map or a circle map with non-empty
set of periodic points. Then any invariant probability measure $\mu$
for whom there exists a point $x$ with $\mu(\omega_f(x))=1$ can be
approximated by CO-measures arbitrary well. In particular, CO-measures
are dense in all ergodic invariant measures of $f$.
\end{theorem}

For the sake of completeness let us also state the result which
describes maps of compact one-dimensional branched manifolds (abusing
the language we will call them \emph{graphs} from now on) which do
\emph{not} have periodic points (see \cite{ak79} for the circle and
\cite{blok84, blok86, blok87a, blok87b} for maps of any graph). Observe
that we do not assume our graphs to be connected; also, to avoid trivialities
let us assume that our maps are \emph{onto} (otherwise we can simply
consider the nested sequence of images of the space and take their intersection which will still
be a graph and on which the map will be onto). A natural
one-dimensional map without periodic points is an irrational circle
rotation. An extension of that is a map which permutes (not necessarily
cyclically) a finite collection of circles so that this collections
falls into several cycles of circles and in each cycle the appropriate
power of the map fixes the circles and acts on each of them as an
irrational rotation (it is easy to see that then in each cycle of
circles it is the rotation by the same irrational angle, but these angles may
change from a cycle to a cycle). Let us call
such maps \emph{multiple irrational circle rotations}.

It turns out that multiple irrational circle rotations are prototypes
of all graph maps without periodic points. By a \emph{monotone} map of
a topological space we mean a map such that point-preimages (sometimes
called \emph{fibers}) are connected.

\begin{theorem}[\cite{blok84}]\label{t:no-per} Suppose that $f:X\to X$ is a continuous
map of a graph to itself with no periodic points. Then there exists a
monotone map from $X$ to a union $Y$ of several circles which
semiconjugates $f$ and a multiple irrational circle rotation.
\end{theorem}

In Subsection \ref{ss:fun-rot} we explained that given a degree one circle map $f:\uc\to \uc$,
one can define the function $\phi_f$ by choosing a lifting $F:\R\to
\R$, then for any $x\in \uc$ a lifting $X$ of $x$, and then setting
$\phi_f(x)=F(X)-X$ so that $\phi_f$ which is well-defined and continuous (because
$f$ is of degree one and continuous). Classical rotation numbers and sets of
points of $\uc$ under $f$ are in fact $\phi_f$-rotation numbers and sets.
Evidently, Theorem~\ref{t:codense}, the definition of weak topology on
probability invariant measures of $f$, and the definition of the
classical rotation numbers and sets imply that rotation numbers of
periodic points are dense in the rotation set of a circle map $f:\uc\to
\uc$ of degree one provided $f$ has some periodic points. Actually, the
fact that the classical rotation set of a degree one circle map is a
closed interval can also be deduced from the ``spectral
decomposition'', however this goes way beyond the scope of the present
paper.

The situation with over-rotation numbers is similar but slightly more
complicated. The issue here is that for over-rotation numbers, the
dynamics in small neighborhoods of fixed points can play a misleading
role. To explain this, let us draw analogy with the case of the
topological entropy (see \cite{akm65} where the concept was introduced
and \cite{alm00, dgs76} for a detailed description of its properties).
It is known that for continuous interval maps it can happen so that the
entropy of such maps is large (even infinite) while it is assumed on
smaller and smaller invariant sets converging to fixed points of the
map. Similarly, it can happen that the dynamics in a small neighborhood
of, say, an attracting fixed point $a$ is chaotic in the sense that
points ``switch sides'', i.e. map from the left of $a$ to the right of
$a$, in a chaotic fashion while still being attracted to $a$. That may
lead to a rich set of sequences $\chi(f^i(x))$ and large
$\chi$-rotational sets of such points while having no bearing upon the
set of periodic points of the map at all. To avoid this ``artificial''
richness we consider only {\it admissible} points.

Namely, by a {\it limit measure} of a point $x$ we mean a limit of
ergodic averages of the $\delta$-measure concentrated at $x$; clearly,
any limit measure is invariant \cite{dgs76}. If $\mu$ is a unique limit
measure of $x$, then $x$ is said to be \emph{generic} for $\mu$. Call a
point $x$ {\it admissible} if any limit measure $\mu$ of $x$ is such
that $\mu(Fix(f))=0$ where $Fix(f)$ is the set of all fixed points of
$f$; since $\mu$ is invariant, this implies that in fact the set of all
points $x$ which are eventual preimages of fixed points of $f$ is of
zero $\mu$-measure too. Since the set of discontinuities of $\chi$ is
contained in the union of the set of fixed points $Fix(f)$ of $f$ and
their preimages, we see that for an admissible point $x$ the set of
discontinuities of $\chi$ is of zero limit measure for any limit
measure of $x$. Now, let $x$ be an admissible point. Take a number
$u\in I_{f, \chi}(x)$; the definitions and properties of measures imply
that $u=\int \chi(x)\,d\mu$ where $\mu$ is a limit measure of $x$. By
Theorem~\ref{t:codense} and by definitions $\mu$ can be approximated
arbitrarily well by a CO-measure concentrated on a non-fixed periodic
orbit. Hence $u$ can be approximated arbitrarily well by the
over-rotation number of a non-fixed periodic orbit, and so $u\in I_f$.
Thus, $I_{f, \chi}(x)\subset I_f$ as long as $x$ is admissible (see
Theorem~\ref{t:b1}).

Additional arguments allow us to prove Theorem~\ref{t:b1}, describing
the connection between $I_f$ and the pointwise $\chi$-rotation sets
$I_{f, \chi}(x)$.

\begin{theorem}[\cite{B1, B3}]\label{t:b1} The following statements are true.
\begin{enumerate}
\item If $f$ is continuous and $\rho_f<1/2$ then for any $a\in
    (\rho_f, 1/2]$ there is an admissible point $x$, generic for a
    measure $\mu$, such that $I_f(x)=\{a\}$.
\item If $x$ is an admissible point, then $I_{f, \chi}(x)\subset
    I_f=[\rho_f, 1/2]$.
\item If $f$ is piecewise-monotone and $\rho_f\ne 0$ then there
    exists an invariant measure $\mu$ such that $f$ is minimal on
    the support of $\mu$ and there exists a point $x$, generic for
    $\mu$ and such that $I_{f, \chi}(x)=\{\rho_f\}$.
\end{enumerate}
\end{theorem}

Theorem~\ref{t:b1}(3) cannot be extended for all continuous interval
maps as one can design a map $f$ which has a sequence of invariant
intervals with their ``own'' maps that have increasing to $[u,
\frac12]$ over-rotation intervals; evidently, for such a map $f$ the
conclusions of Theorem~\ref{t:b1} do not hold.

In a recent paper by Jozef Bobok [Bo] the case covered in Theorem \ref{t:b1}(3)
is studied in great detail and depth resulting into a much more precise
claim. Recall that a dynamical system is said to be {\it strictly
ergodic} if it has a unique invariant measure. To state Theorem \ref{t:bo1} in
full generality we need a couple of notions on which we will elaborate
later in Subsection~\ref{ss:combi1}. Namely, a cyclic permutation $\pi$
{\it forces} a cyclic permutation $\theta$ if a continuous interval map
$f$ which has a cycle inducing $\pi$ always has a cycle inducing
$\theta$. By \cite{Ba} forcing is a partial ordering. One can talk
about the {\it over-rotation pair} $orp(\pi)$ and the {\it
over-rotation number} $\rho(\pi)$ of a cyclic permutation $\pi$. We
call a cyclic permutation $\pi$ an {\it over-twist} if it does not
force other cyclic permutations of the same over-rotation number.

\begin{theorem}[\cite{Bo}]\label{t:bo1} Let a point $x$ and a measure $\mu$ be
as defined in Theorem~\ref{t:b1}(3). Then the map $f|_{\omega(x)}$ is
strictly ergodic with $\mu$ being the unique invariant measure of
$f|_{\omega(x)}$. Moreover, if $\rho_f$ is rational then $x$ is
periodic and can be chosen so that the permutation induced by the orbit
of $x$ is an over-twist of over-rotation number $\rho_f$.
\end{theorem}

The above results allow one to make conclusions about the dynamics of
an interval map based upon little information: if one knows $\rho_f$
then one can, e.g., describe all possible over-rotation numbers of
$f$-periodic points (except, in the non-piecewise monotone case, the
number $\rho_f$ itself if it is rational). In other words, numerical
information about a map, compressed to $I_f$, implies various types of
the limit behavior of periodic points reflected by their rotation
numbers. Can one say more? In particular, can we explicitly describe
at least some permutations induced by periodic orbits of $f$? By definition the
affirmative answer can be given if one can explicitly describe the
over-twists of given over-rotation numbers. In addition, it is
important to design a practical approach (an algorithm) to figuring out what the
over-rotation interval of a map $f$ is. 

In this paper we address these issues for bimodal interval maps of the
type ``increasing-decreasing-increasing'' (so-called \emph{$N$-bimodal
maps} or bimodal maps \emph{of type N}). The paper develops ideas from
\cite{BS} (the results of \cite{BS} are described in
Section~\ref{s:prelim}). In particular, one of the tools used in
\cite{BS} was a special disconnected conjugacy of a unimodal map to a
discontinuous map of the interval which can be lifted to the degree one
discontinuous map of the real line; in the present paper we show that
this tool apply to a wider class of maps, including $N$-bimodal ones.

Our paper is divided into sections as follows:

\begin{enumerate}

\item Section 2 contains preliminaries.

\item In Section 3 we will show that, given an $N$-bimodal map $f$ we can
    construct its lifting to a degree one map of the real line
    which admits a continuous monotonically increasing lower bound
    function $G$ whose classical rotation number gives us the
    left endpoint of the over-rotation interval of $f$.

\item In Section 4, as an application we will describe the bimodal
    permutations which are forcing-minimal among all permutations
    with the same over-rotation number (i.e., $N$-bimodal
    \emph{over-twist permutations}).

\item In Section 5, we describe a general class of continuous maps,
    called \textit{well behaved maps}, for which a construction
    similar to the one from Section 3 goes through allowing for
    finding the orbit on which the left endpoint of the
    over-rotation interval is assumed. If $f$ is a \textit{map}
    like that and the over-rotation interval of $f$ is $I_f =
    [\frac{p}{q}, \frac{1}{2}]$ where $ \frac{p}{q} \in \mathbb{Q},
    p,q \in \mathbb{Z}, g.c.d(p,q)=1, q\neq 0$ then our
    construction gives a transparent prescription as to where a
    periodic orbit $x$ of $f$ with over-rotation number
    $\frac{p}{q}$ must be located.

\end{enumerate}


\section{ PRELIMINARIES }\label{s:prelim}

This section is divided into short subsection devoted to certain topics
in one-dimensional dynamical systems.

\subsection{Combinatorial dynamics in one-dimension}\label{ss:combi1}

We need definitions from {\it one-dimensional combinatorial dynamics}
(\cite{alm00}). A map $f$ has a {\it horseshoe} if there are two closed
intervals $I, J$ with disjoint interiors whose images cover their
union. In particular, $f$ has a horseshoe if there exist points $a,b,c$
such that either $f(c)\le a=f(a)<b<c\le f(b)$ (set $I=[a, b], J=[b,
c]$) or $f(c)\ge a=f(a)>b>c\ge f(b)$ (set $I=[b, a], J=[c, b]$). It is
easy to see \cite{BM1} that if a map has a horseshoe then it has
periodic points of all possible over-rotation numbers. A {\it (cyclic)
pattern} is the family of all cycles on the real line that induce the
same cyclic permutation of the set $T_n=\{1,2,\dots,n\}$ or its flip; a
map (not even necessarily one-to-one) of the set $T_n$ into itself is
called a {\it non-cyclic pattern}. If one considers the family of all
cycles on the real line that induce the same cyclic permutation (i.e., one
does not allow for a flip), this family is called a \emph{cyclic \textbf{oriented} pattern}.
If an interval map $f$ has a cycle
$P$ from a pattern $\Pi$ associated with permutation $\pi$, we say that
$P$ is a {\it representative} of $\pi$ in $f$ and $f$ {\it exhibits}
$\pi$ (on $P$); if $f$ is {\it monotone (linear)} on each complementary
to $P$ interval, we say that $f$ is {\it $P$-monotone ($P$-linear)}
\cite{MN}. In what follows the same terminology will apply to
permutations, patterns and cycles, so for brevity we will be
introducing new concepts for, say, permutations. Observe also, that permutations
are understood up to orientation. Finally, notice that in what follows
we will interchangeably talk about permutations and patterns.

A permutation $\pi$ is said to have a {\it block structure} if there is
a collection of pairwise disjoint segments $I_0, \dots, I_k$ with
$\pi(T_n\cap I_j)=T_n\cap I_{j+1}, \pi(T_n\cap I_k)=T_n\cap I_0$; the
intersections of $T_n$ with intervals $I_j$ are called {\it blocks} of
$\pi$. A permutation without a block structure is said to be with
\emph{bo block structure}, or, equivalently, {\it irreducible}. If we
collapse blocks to points, we get a new permutation $\pi'$, and then
$\pi$ is said to have a block structure {\it over $\pi'$}. A
permutation $\pi$ {\it forces} a permutation $\theta$ if any continuous
interval map $f$ which exhibits $\pi$ also exhibits $\theta$. By
\cite{Ba} forcing is a partial ordering. If $\pi$ has a block structure
over a pattern $\theta$, then $\pi$ forces $\theta$. By \cite{MN} for
each permutation $\pi$ there exists a unique irreducible pattern $\pi'$ over
which $\pi$ has block structure (thus, $\pi'$ is forced by $\pi$).

The following construction is a key ingredient of one-dimensional
combinatorial dynamics. Let $\pi$ be a (non-cyclic) permutation, $\Pi$
be its pattern, $P$ be a finite set with a map $f:P\to P$ of pattern
$\Pi$, and $f$ be a $P$-linear map; assume also that the convex hull of
$P$ is $[0, 1]$. Say that the closure $\ol{I}$ of a component $I$ of
$[0, 1]\setminus P$ {\it $\pi$-covers} the closure $\ol{J}$ of another
such component $J$ if $\ol{J}\subset f(\ol{I})$. Construct the oriented
graph $G_\pi$ whose vertices are closures of the components of $[0,
1]\setminus P$ and whose edges (arrows) go from $\ol{I}$ to $\ol{J}$ if
and only if $\ol{I}$ $\pi$-covers $\ol{J}$. Clearly, $G_\pi$ does not
depend on the actual choice of $P$ and the definition is consistent.

A cycle is {\it divergent\/} if it has points $x<y$ such that $f(x)<x$
and $f(y)>y$. A cycle that is not divergent will be called {\it
convergent\/}. It is well-known that if a pattern is divergent then for
any cycle $P$ of this pattern the $P$-linear map has a horseshoe; on
the other hand, it is easy to see that if a pattern is convergent then
a cycle $P$ of this pattern cannot give rise to the $P$-linear map with
a horseshoe.

\subsection{Rotation theory on the interval}\label{ss:rot1}

One can talk about the {\it over-rotation pair} $orp(\pi)$ and the {\it
over-rotation number} $\rho(\pi)$ of a permutation $\pi$. We call a
permutation $\pi$ an {\it over-twist permutation} (or just an {\it
over-twist}) if it does not force other permutations of the same
over-rotation number; the pattern of an over-twist permutation is said
to be an \emph{over-twist} pattern. Theorem~\ref{t:bm2} and the
properties of forcing imply the existence of over-twist patterns of any
given rational over-rotation number between $0$ and $1$; in fact, it
implies that a map which has a periodic point of rational over-rotation
number $\rho$ exhibits an over-twist pattern of rotation number $\rho$.
By Theorem~\ref{t:bm2} an over-twist pattern has a coprime
over-rotation pair; in particular, over-twists of over-rotation number
$1/2$ are of period $2$, so from now on we consider over-twists of
over-rotation numbers distinct from $1/2$.

Suppose that $\pi$ is a convergent pattern and that $P$ is a periodic
orbit of pattern $\pi$. Let $f$ be a $P$-linear map. Then $f$ has a
unique fixed point $a$. Consider the set $Q=P\cup \{a\}$ and denote its
pattern by $\pi'$. We will work with the oriented graph $G_{\pi'}$.
Suppose that there is a real-valued function $\psi$ defined on arrows
of $G_{\pi'}$. It is well-known \cite{alm00} that the maximal and the
minimal averages of $\psi$ along all possible paths (with growing
lengths) in $G_{\pi'}$ are assumed, in particular, on periodic
sequences. If the values of $\psi$ on arrows are all rational, then the
maximum and the minimum of those averages are rational too.

We choose a specific function $\psi$ as follows. Associate to each
arrow in $G_{\pi'}$ the number $1$ if it corresponds to the movement of
points from the right of $a$ to the left of $a$. Otherwise associate
$0$ to the arrow. As explained above, this yields rational maximum and
rational minimum of limits of averages of $\psi$ taken along all
possible paths (with growing lengths) in $G_{\pi'}$, and these extrema
are assumed on periodic sequences. Given a cycle of $f$, one can
consider its cycles and compute out for them their over-rotation
numbers; simultaneously, $\psi$-rotation numbers can be computed out
for the associated paths (loops) in the oriented graph $G_{\pi'}$.
Evidently, the over-rotation numbers of $f$-cycles and the
$\psi$-rotation numbers of the associated loops in $G_{\pi'}$ are the
same.

We also need to introduce some classical concepts.

\begin{definition}\label{d:deg1}
A map $F:\R\to \R$ is said to be of degree $1$ if $F(x+1)-F(x)=1$ for
any $x\in \R$.
\end{definition}

Classical results of Poincar\'e \cite{poi} apply to all monotonically
increasing maps of the real line of degree one \cite{rt86} for whom
every point $y\in R$ has the same classical rotation number defined as
the limit of the sequence $F^n(x)/n$. 	


\section{$N$-BIMODAL MAPS}\label{s:nbimo}

Given a rational number $\rho, 0<\rho\le 1/2$, we want to describe all
over-twist patterns of $N$-bimodal type of over-rotation number $\rho$.
In the beginning of this section we outline our approach to the
problem.

Our arguments  are based upon an extension of a construction from
\cite{BS} onto the $N$-bimodal case. This gives rise to a special
lifting of a given $N$-bimodal interval map $f$ to a degree one
\emph{discontinuous} map $F_f$ of the real line. The construction is
designed to guarantee that over-rotation numbers of $f$-periodic points
of the interval coincide with the classical Poincar\'e rotation numbers
of the corresponding points under $F_f$. Even though the classical tools
of \cite{mis82} do not apply to $F_f$ (after all, $F_f$ is discontinuous),
our construction implies the existence of a \emph{continuous
non-strictly monotonically increasing} function $G_f\le F_f$ with important
properties.

Namely, the monotonicity of $G_f$ implies that $G_f$ is semiconjugate
to a circle map $\tau_f:\uc\to \uc$ and $\tau_f$ is monotone (i.e.,
point preimages under $\tau_f$ are connected) circle map which is
either locally constant or locally increasing (we consider
counterclockwise direction on $\uc$ as positive). One can define the
$G_f$-rotation number for every point $y\in \mathbb R$, and for all
$y$'s this number will be the same; denote it by $\rho'_f$. The set
$A_f$ of points $y$ such that $G_f(y)<F_f(y)$ is, evidently, open; it
follows from the construction, that $G_f$ is a constant on each
component of $A_f$. It is well-known that then there exist points $x$
whose $G_f$-trajectories avoid $A_f$, and if $\rho'_f$ is rational then
there exists a point $x$ on which $G_f$ acts so that the associated to
$x$ points $x'\in [0, 1]$ and $x''\in \uc$ are both periodic.
Evidently, the classical rotation pair of $x''$ and the over-rotation pair of
$x'$ coincide. This implies that the over-rotation pair of $x'$ is coprime.

Moreover, we have that (1) $G_f\le F_f$, (2) the rotation numbers of
points of $F_f$ equal the over-rotation numbers of the corresponding
points of $[0, 1]$, (3) on $x$ and, inductively, on all its images, we
have $F_f=G_f$, and (4) the classical rotation number $\rho'_f$ of $G_f$
can be computed on the orbit of $x$. This implies that the left
endpoint $\rho_f$ of the over-rotation interval of $f$ equals $\rho'_f$ and that it
is assumed on the $f$-orbit of $x'$. We use results of \cite{bb19} to
deduce then that the pattern of $x'$ is an over-twist. The classical
rotation numbers of $x$ in the sense of $F_f$ and in the sense of $G_f$
are the same because the maps are the same on the trajectory of $x$.
Any $f$-periodic point $y\in [0, 1]$ has its over-rotation number equal
to the rotation number of $y$ in the sense of $F_f$; since $G_f\le
F_f$, it is greater than or equal to that of $x$ (on whose trajectory
$G_f=F_f$). Hence on the trajectory of $x$ the over-rotation number is
minimal among all cycles of $f$.

\begin{definition}\label{d:nbimo}
By an $N$-bimodal map we shall mean a continuous map $f:[0,1]\to [0,1]$
which satisfies the following properties:

\begin{enumerate}

\item  f has a unique fixed point $a_f$, a point of local maxima
    $M_f$, and a point of local minima $m_f$ such that $M_f<a_f<m_f$;

\item $f(M_f)= 1$ and  $f(m_f)=0$.

\end{enumerate}

\end{definition}

If $f(0)>a_f>f(1)$ then it is easy to see that the interval $[0, a_f]$
maps into the interval $[a_f, 1]$ and vice versa. Clearly, the
over-rotation interval of $f$ is degenerate and coincides with
$\{1/2\}$. We consider this case as trivial and do not deal with it in
the rest of the paper. Thus, from now we assume that the point $a_f$
has preimages from at least one side; we may assume that $f(1)\ge a_f$
and, therefore, there is always a preimage $a''_f$ of $a_f$ with
$m_f>a''_f>a_f$. We will also use the following notation. Let $d_1(f)$
be the unique point in the interval $(M_f, a_f)$ such that
$f(d_1(f))=f(1)$. Clearly, $d_1(f)$ exists by the intermediate value
theorem. Similarly, let $d_2(f)$ be the unique point in the interval
$(a_f, m_f)$ such that $f(d_2(f))= f(0)$ ($d_2(f)$ is actually defined
only if $f(0)\le a_f$, otherwise we assume $d_2(f)$ to be undefined).The $N$-bimodal map in Case$1$ and Case$2$ are shown in Figure \ref{Nbimodal:Case1} and \ref{Nbimodal:Case2} respectively.

\begin{figure}[H]
	\caption{\textbf{\textit{An N-bimodal map $f$ in case when $d_2(f)$ is defined}}}
	\centering
	\includegraphics[width=0.55\textwidth]{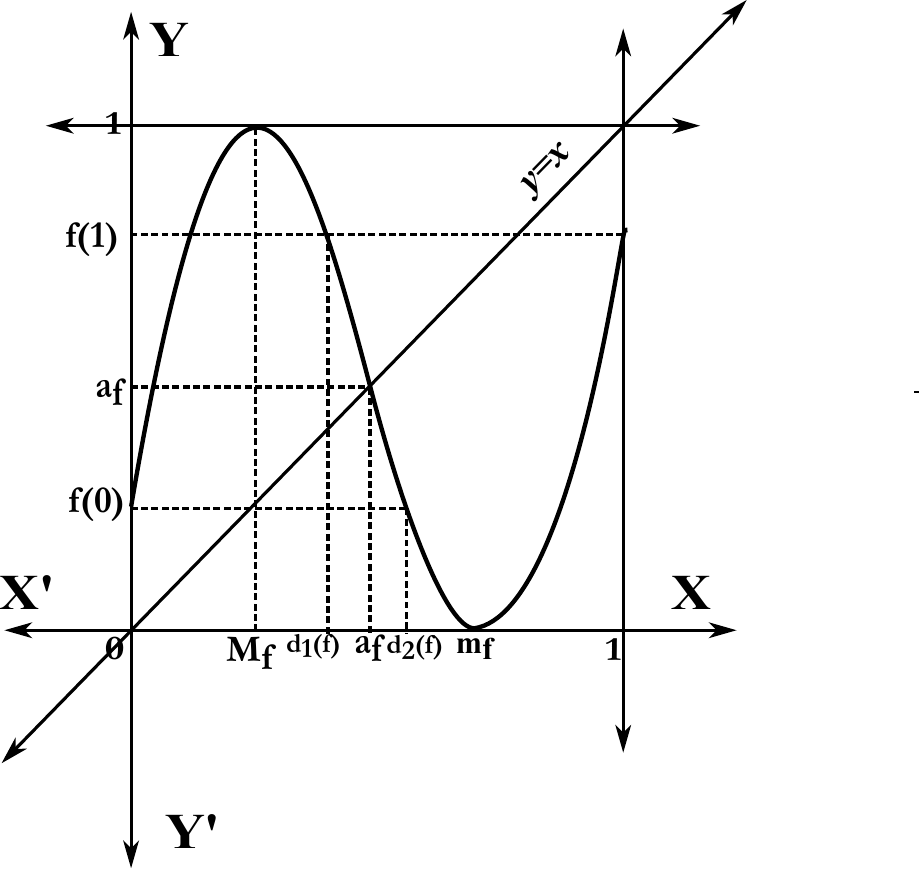}
	\label{Nbimodal:Case1}
\end{figure}

\begin{figure}[H]
	\caption{\textbf{\textit{An N-bimodal map $f$ in case $d_2(f)$ is not defined}}}
	\centering
	\includegraphics[width=0.55\textwidth]{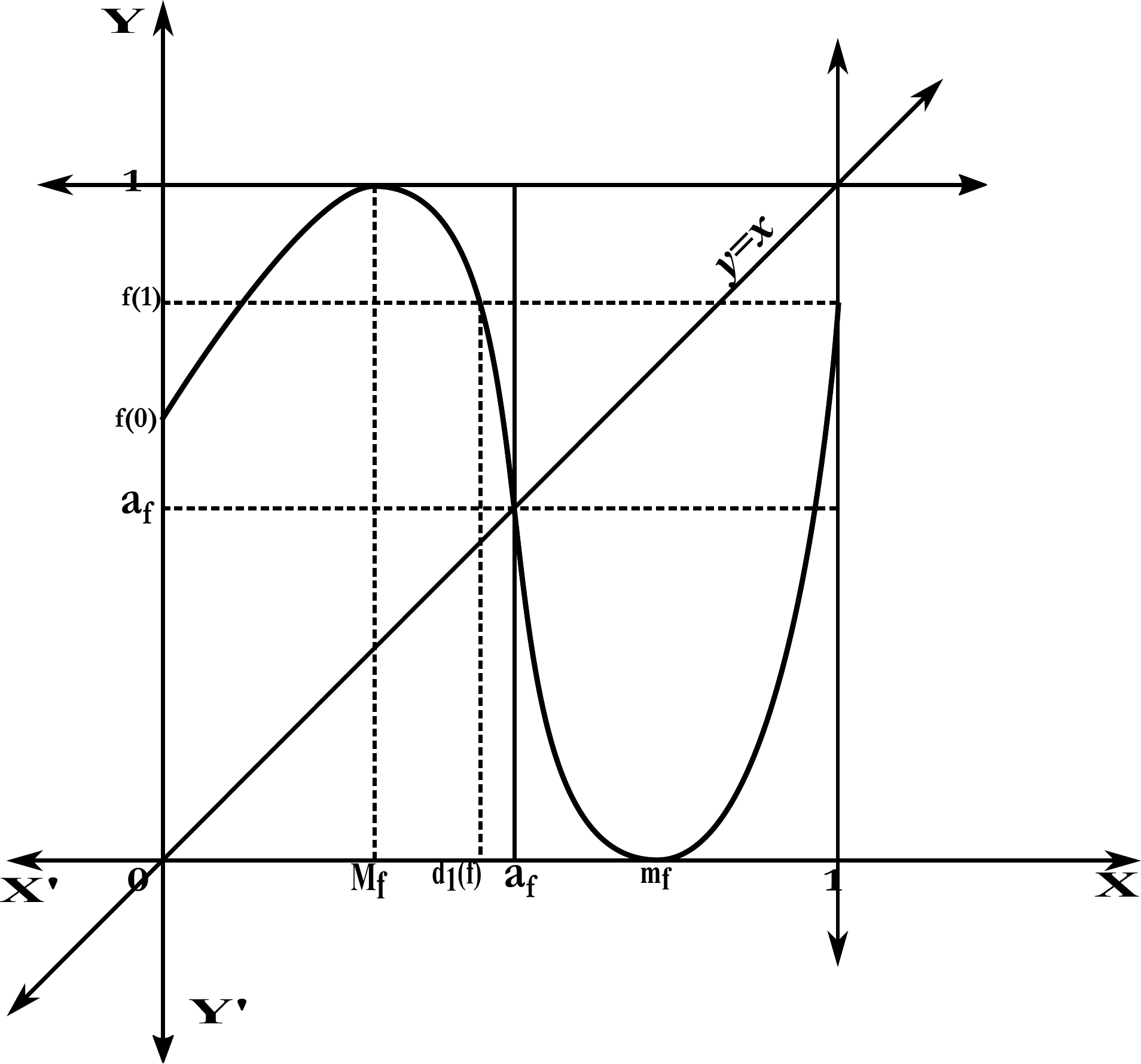}
	\label{Nbimodal:Case2}
	
\end{figure}

\subsection{Disconnected lifting of an $N$-bimodal map
$f$}\label{ss:discon}

In this subsection we construct a few maps all of which are based upon a
given $N$-bimodal map $f$. For the sake of simplicity of notation we often
are not using subscript $f$; this is justified because at this point
there are no other maps and, therefore, omitting the subscript will not
lead to ambiguity. However in the future we may occasionally use
subscripts to emphasize dependence of our construction upon a given
map.

Consider the map $g_f=\sigma_f \circ f \circ \sigma_f^{-1}$ where the map
$\si_f$ is defined below. However, before defining $\si$ we need to make
an observation. Our aim is to study periodic \emph{non-fixed} points of
$f$ and their over-rotation numbers; we do it by working with periodic
\emph{non-fixed} points of $g_f$. Therefore the behavior of $f$ at $a_f$
is not important for us. This allows us to ignore the fact that with
the adopted below definitions the map $\si_f$ and, as a result, the
function $g_f$ is multivalued at $a_f$ and its preimages (indeed, this
does not have any bearing upon the over-rotation numbers of
\emph{non-fixed} $f$-periodic points as they never map to $a_f$). Let
us now denote the over-rotation interval $I_f$ of $f$ by $I_f=[\mu,
\frac{1}{2}]$, and consider a discontinuous conjugacy $\sigma_f :[0,1]\to
[0,1]$ defined by

\begin{equation}
\sigma_f(x)=
\begin{cases}
				
x & \text{if $0\le x\le a_f$}\\
a_f+1-x  &   \text{if $a_f\le x\le 1$}\\
				
\end{cases}
\end{equation}

The map $\si_f$ flips the interval $[a_f, 1]$ symmetrically with respect
to the midpoint $\frac{1+a_f}{2}$ of $[a_f, 1]$ so that $\si_f^2$ is the
identity on the entire $[0, 1]$. We now define the map $g_f:[0,1]\to
[0,1],$ $g=\sigma_f \circ f \circ \sigma_f^{-1}$ to which $\si_f$ conjugates
the map $f$. In what follows by $\si_f'$ we mean the map $\si_f$ restricted
upon $[a_f, 1]$; moreover, if we flip points of the plane in the
vertical direction with respect to the line $y=\frac{1+a_f}{2}$ we
shall say that we apply \emph{vertical $\si_f'$}, and if we flip points
of the plane in the horizontal direction with respect to the line
$x=\frac{1+a_f}{2}$ we shall say that we apply \emph{horizontal
$\si_f'$}.

\medskip

\noindent \emph{Case 1: $f(0)\le a_f.$}\, Then $a_f$ has two preimages,
$a_f' $ and $a_f''$, and we have $0<a_f'<M_f<a_f<m_f<a_f''<1$. Let us
now describe the graph of the function $g_f$ by giving the expression for
$g_f(x)$ depending on the location of $x$.

(a) On the interval $[0, a_f']$, $g_f(x)=f(x)$, that is, the graph of $g_f$
is the same as the graph of $f$ in the interval $[0, a_f']$.

(b) On the interval $[a_f',a_f]$, $\sigma_f(x)=x$ and $f(\sigma_f(x))\ge
a_f$ so that $g_f(x)=a_f + 1 - f(x)$; the graph of $g_f$ is obtained from
that of $f$ by applying the vertical $\si_f'$ to it.

(c) On the interval $[a_f, a_f+1-a_f'']$, $g_f(x)= a_f + 1 - f(a_f +
1-x)$. In other words, this part of the graph of $g_f$ can be obtained
from the part of the graph of $f$ located above the interval $[a''_f,
1]$ by first applying the horizontal $\si_f'$ to it, and then applying
the vertical $\si_f'$ to it.

(d) On the interval $[a_f+1-a_f'',1]$ , $g_f(x)=f(a_f+1-x)$. So, the
graph of g in the interval, $[a_f,1]$ can be obtained from the graph of
$f$ located above $[a_f, a'_f]$ by applying the horizontal $\si_f'$ to
it.

It immediately follows from the definitions, that $I_f=I_g$. Hence in
studying over-rotation numbers of periodic points we can concentrate
upon the map $g_f$. We do so by defining a \emph{degree one} lifting $F_f$
of $g_f$ to the real line. The lifting is designed so that the classic
rotation numbers of points of $F_f$ corresponding to periodic points of
$g_f$ in fact equal over-rotation numbers of these periodic points of
$g_f$.

Here is how we define a degree one lifting $F_f:\mathbb{R}\to
\mathbb{R}$;the idea is to keep $F_f=g_f$ everywhere except for the
interval $[a_f+1-a_f'', 1]$ located to the right of $a_f$ on which
points are mapped to the left of $a_f$ so that when we compute out the
corresponding over-rotation number the number $1$ should be added:

\begin{equation}
F_f(x)=
\begin{cases}
			
g_f(x)=f(x) & \text{if $0\leq x \leq a_f'$}\\
g_f(x)=a_f+1-f(x)  &   \text{if $a_f' \leq x \leq a_f$}\\
g_f(x)=a_f+1-f(a_f+1-x) & \text{if $a_f \leq x <a_f+1-a_f''$ }\\
g_f(x)+1= 1+ f(a_f+1-x) & \text{if $a_f+1-a_f'' \leq x\le 1$}\\
\end{cases}
\end{equation}

Then we extend $F_f$ onto the real line as a degree one map; this means
that on each $[n, n+1] $ where $n \in \mathbb{Z}$ we set
$F_f(x+n)=F_f(x)+n$ for all $n \in \mathbb{Z}$ and $x \in [0,1]$.
The graph of the $F_f$ is shown in Figure \ref{F:Case1}.

\begin{figure}[H]
	\caption{\textbf{\textit{Construction of the map  $F_f$ for an N-bimodal map $f$ in Case 1}}}
	\centering
	\includegraphics[width=1\textwidth]{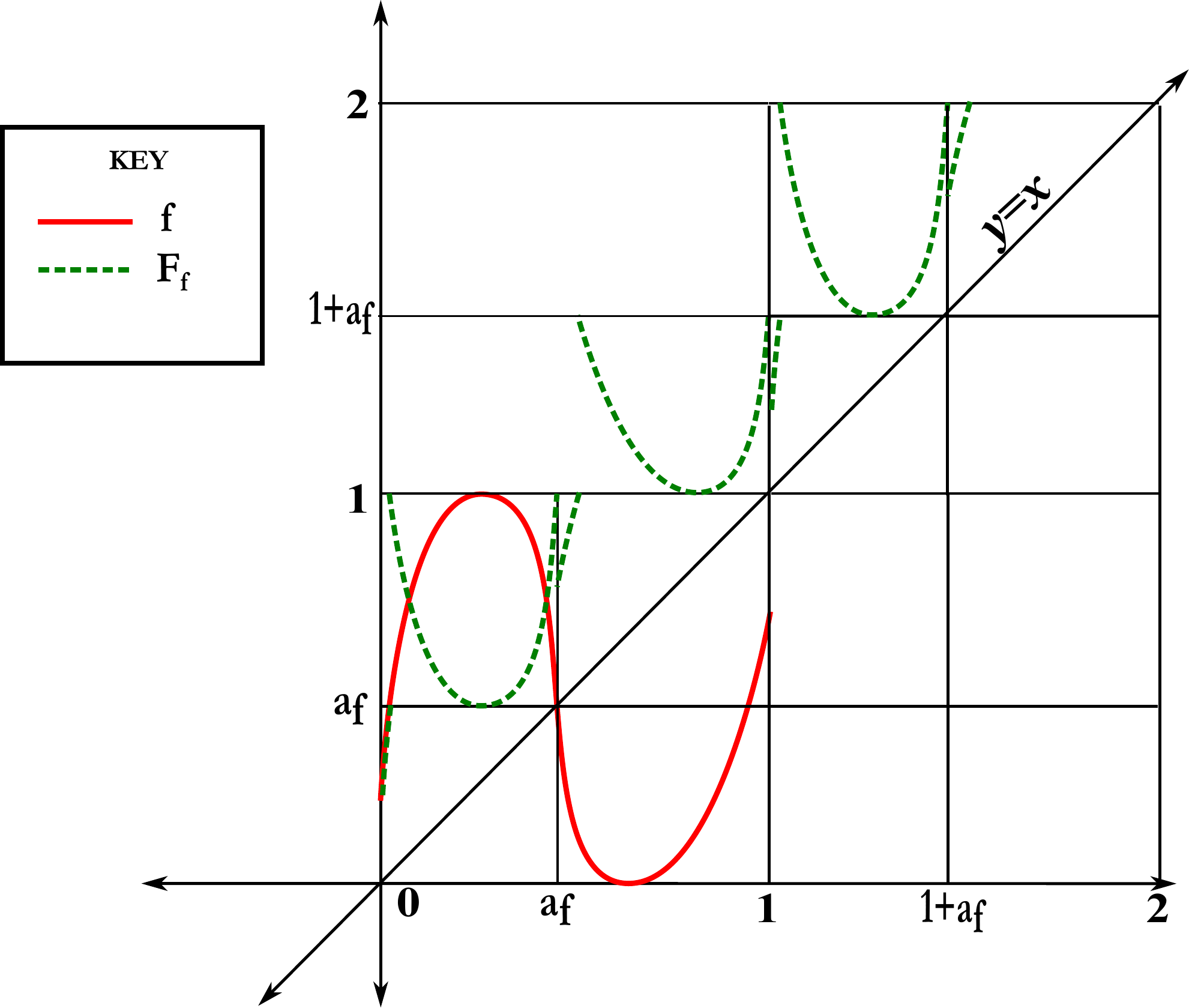}
	\label{F:Case1}
\end{figure}

The last step in this series of maps is a continuous map $G_f:
\mathbb{R}\to \mathbb{R}$.

\begin{figure}[H]
	\caption{\textbf{\textit{Construction of the map  $G_f$ for an N-bimodal map $f$ in Case 1}}}
	\centering
	\includegraphics[width=1.2\textwidth]{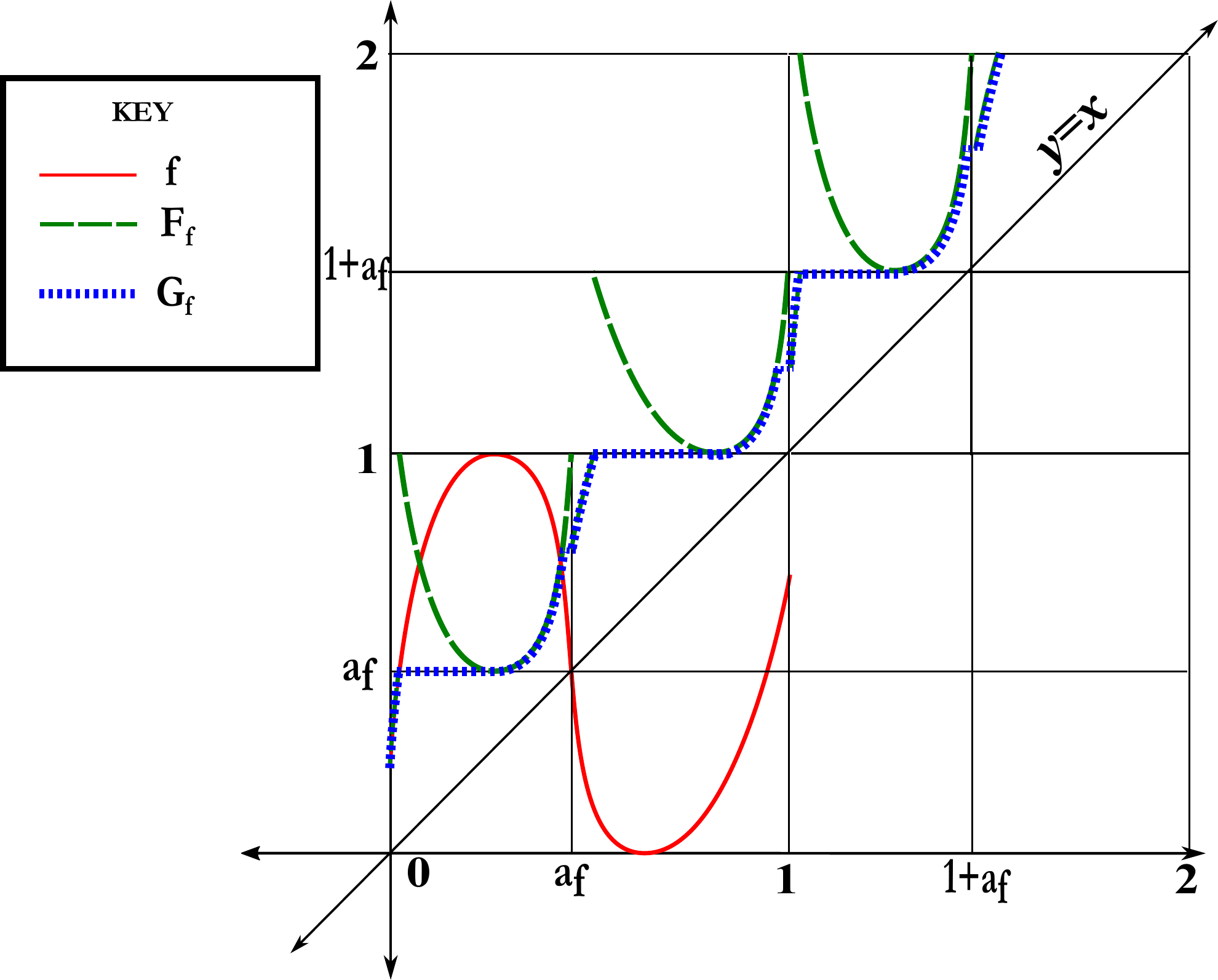}
	\label{G:Case1}
\end{figure}

The map $G_f$ is defined on each $[n, n+1], n \in \mathbb{Z}$
in the following fashion:

\begin{equation}
G_f(x)=
\begin{cases}

F_f(x) & \text{if $n\leq x \leq n+a_f'$}\\
n+a_f  &   \text{if $n+a_f' \leq x \leq n+M_f$}\\
F_f(x)  & \text{if $n+M_f \leq x \leq n+d_1(f)$}\\
F_f(d_1) & \text{if $n+d_1(f) \leq x \leq n+a_f$ }\\
F_f(x) & \text {if $n+a_f \leq x \leq n+a_f+1-a_f''$}\\
n+1 & \text {if $ n+a_f+1-a_f'' \leq x \leq n+a_f+1-m_f$ }\\
F_f(x) & \text {if $ n+a_f+1-m_f \leq x \leq n+a_f+1-d_2(f)$}\\
F_f(n)+1 & \text{if $n+a_f +1- d_2(f) \leq x \leq n+1$}\\

\end{cases}
\end{equation}

The graph of $G_f$ is shown in the Figure \ref{G:Case1}.
						
\noindent \emph{Case 2: $f(0)> a_f.$}\, In this case $a_f$ has only one
preimage, namely $a_f''$, and we have $0<M_f<a_f<m_f<a_f''<1.$ The
points $a'_f$ and $d_2(f)$ are undefined. The functions $g_f,$ $F_f$ and
$G_f$ will be slightly different. On the interval $[0,a_f]$, $g_f(x)= a_f +
1 - f(x)$, i.e. the graph of $g_f$ can be obtained from the graph of $f $
by applying vertical $\si_f$. The function $g_f$ is same as in the earlier
case on the interval $[a_f, 1]$.The map $F_f:\mathbb{R}\to \mathbb{R}$
is now defined as follows:

\begin{equation}
F_f(x)=
\begin{cases}

g_f(x)=a_f+1-f(x)  &   \text{if $ 0 \leq x \leq a_f$}\\
g_f(x)=a_f+1-f(a_f+1-x) & \text{if $a_f \leq x <a_f+1-a_f''$ }\\
g_f(x)+1= 1+ f(a_f+1-x) & \text{if $a_f+1-a_f'' \leq x\le 1$}\\
\end{cases}
\end{equation}

Then, as before, we define $F_f$ on each $[n, n+1],$ $n \in \mathbb{Z}$
as $F_f(x+n)=F_f(x)+n$ for all $n\in \mathbb{Z}$ and $x \in [0,1]$.
 The graph of $F_f$ is shown in Figure \ref{F:Case2}.		
\begin{figure}[H]
\caption{\textbf{\textit{Construction of the map  $F_f$  for an N-bimodal map $f$ in Case 2}}}
				\centering
				\includegraphics[width=1\textwidth]{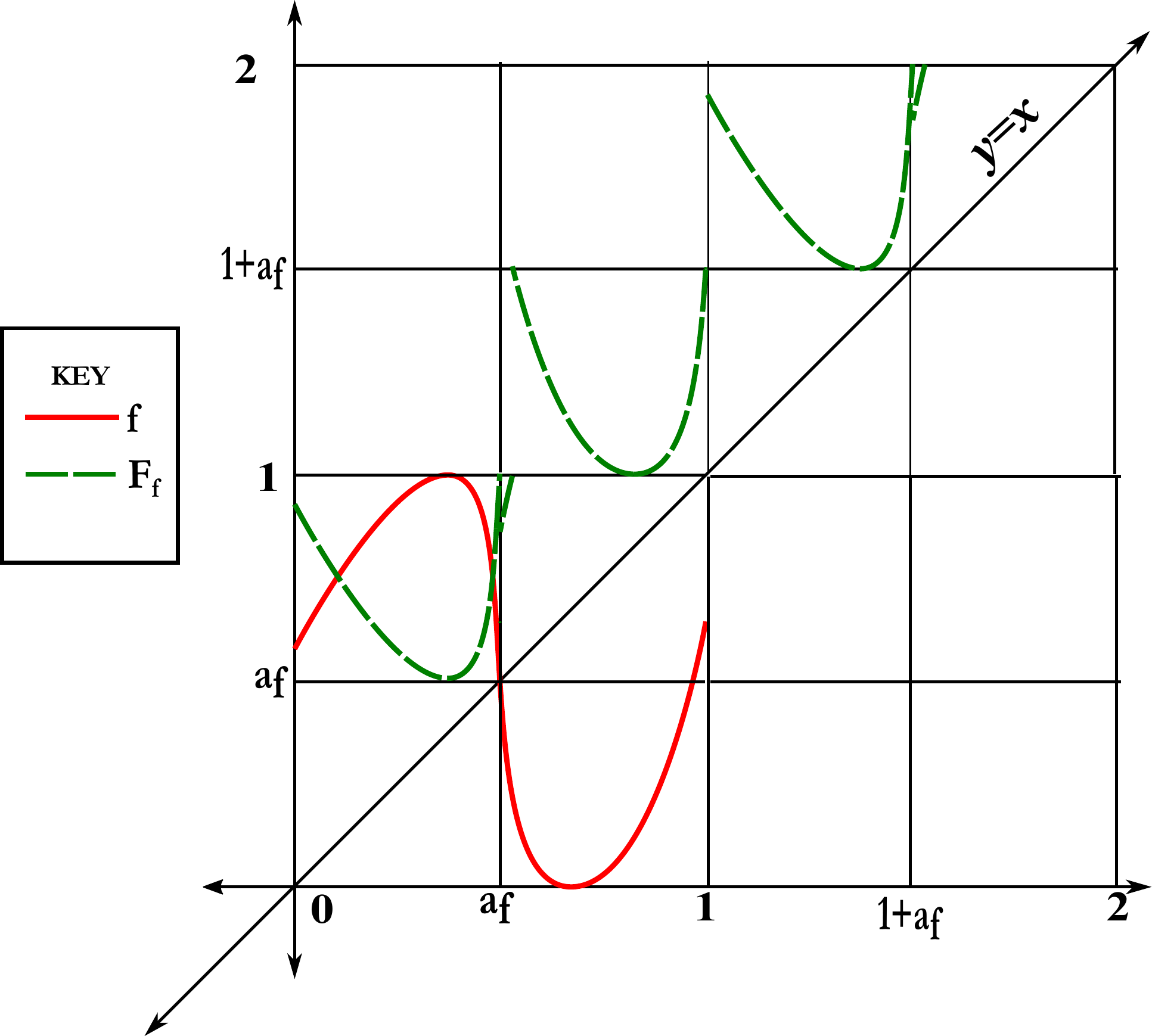}
				\label{F:Case2}
			\end{figure}
	
		On each $[n, n+1],$ $n \in \mathbb{Z}$ the map $G_f:\mathbb{R}\to \mathbb{R} $
will now be defined as follows:

\begin{equation}
G_f(x)=
\begin{cases}

n+a_f  &   \text{if $n \leq x \leq n+M_f$}\\
F_f(x)  & \text{if $n+M_f \leq x \leq n+d_1(f)$}\\
F_f(d_1) & \text{if $n+d_1(f) \leq x \leq n+a_f$ }\\
F_f(x) & \text {if $n+a_f \leq x \leq n+a_f+1-a_f''$}\\
n+1 & \text {if $ n+a_f+1-a_f'' \leq x \leq n+a_f+1-m_f$ }\\
F_f(x) & \text {if $ n+a_f+1-m_f \leq x \leq n+1$}\\

\end{cases}
\end{equation}

The graph of $G_f$ is shown in Figure \ref{G:Case2}.

\begin{figure}[H]
	\caption{\textbf{\textit{Construction of the map  $G_f$ for an N-bimodal map $f$ in Case 2}}}
	\centering
	\includegraphics[width=1.1\textwidth]{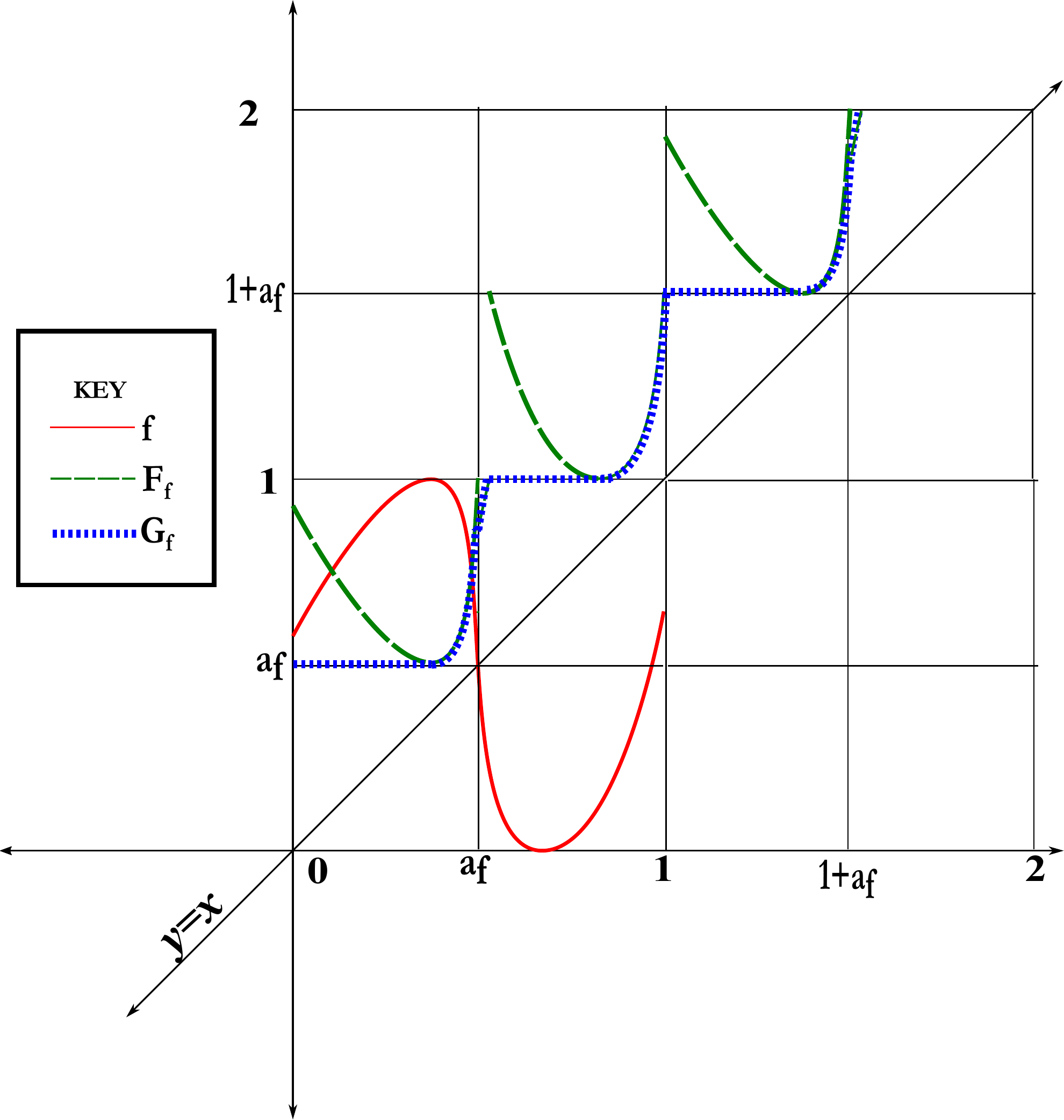}
	\label{G:Case2}
\end{figure}

		In what follows we will consider the relation of the classical
		Poincar\'e rotation numbers of points of the real line under $F_f$ and
		over-rotation numbers of points of $[0, 1]$ in the sense of the map $f$
		(equivalently, the map $g_f$). We reserve the just introduced
		notation for the maps $g_f,$ $F_f$ and $G_f$ (assuming that a map $f$
		is given). Notice, that, as one can see from the above, we will often
		deal with (continuous) functions on the real line that have open
		intervals on which the functions are constants. Let us call
		maximal such interval \emph{flat spots} (of the corresponding
		function). Moreover, the same terminology trivially applies to circle
		maps too.
		
		\subsection{Rotation numbers for $f$ and $F$}\label{ss:rot-f-F}
		We are ready to prove the next theorem that relates the rotation
		numbers of the above constructed maps. By a \emph{minimal} set of a map $h$
		of a compact space to itself we mean an invariant compact set $Z$ all
		of whose points have dense trajectories in $Z$ (thus, a compact
		invariant subset of $Z$ coincides with $Z$). For the map $f$ considered
		in Theorem~\ref{t:bimo-twist} we assume that all agreements and
		notation discussed in the beginning of Section~\ref{s:nbimo} hold; this
		time, however, we will use the subscripts to emphasize the dependence
		of the construction upon a given map. Moreover, it is useful to compare
		this theorem with Theorem~\ref{t:bo1}.

\begin{theorem}\label{t:bimo-twist}
Let $f:[0, 1]\to [0, 1]$  be an $N$-bimodal map. Then the continuous
non-strictly monotonically increasing function $G_f:\mathbb{R}\to
\mathbb{R}$ has the rotation number $\rho'_f$ coinciding with the left endpoint
$\rho_f$ of the over-rotation interval $[\rho_f, \frac12]$ of the map
$f$. Furthermore, there exists a minimal $f$-invariant set $Z_f$ such
that for every point $y\in Z_f$ we have $I_{f, \chi}=\rho_f$, and there
are two possibilities:

\begin{enumerate}

\item $\rho'_f=\rho_f$ is rational, $Z_f$ is a periodic orbit, and
    $f|_{Z_f}$ is canonically conjugate to the circle rotation by $\rho_f$
    restricted on one of its cycles so that the over-rotation pair of $Z_f$ coincides
    with the classical rotation pair of the circle rotation by $\rho_f$
    (in particular, the over-rotation pair of $Z_f$ is coprime) and the over-rotation interval
    of the $Z_f$-linear map is $[\rho_f, 1/2]$;

\item $\rho'_f=\rho_f$ is irrational, $Z_f$ is a Cantor set, and $f|_{Z_f}$
    is canonically at most two-to-one semi-conjugate to the circle
    rotation by $\rho_f$.

\end{enumerate}

\noindent Moreover, define the set $Y_f$ as
\[
Y_f=[0, a'_f]\cup [M_f, d_1(f)]\cup [d_2(f), m_f]\cup [a_f'', 1]
\]

\noindent if $f(0)\le a_f$ and, therefore, $a_f'$ exists and is
well-defined, or as

\[
Y_f=[M_f, d_1(f)]\cup [a_f, m_f]\cup [a''_f, 1]
\]

\noindent if $f(0)>a_f$. Then $Z_f\subset Y_f$ and for every point $y$
whose trajectory is contained in $Y_f$ we have that $I_{f,
\chi}(y)=\{\rho_f\}$.
\end{theorem}

\begin{proof}
In Subsection \ref{ss:discon} we introduced the maps $F_f$ and $G_f$; in doing so, we considered Case and Case 2 depending upon whether $f(0)<a_f$ (Case 1) or $f(0)\ge a_f$ (Case 2). In the proof of this theorem we will consider only Case 1 as Case 2 is completely analogous to Case 2.

Evidently, the classical rotation numbers and rotation pairs of the map
$F_f$ coincide with the over-rotation numbers and over-rotation pairs
of the map $g_f$, and hence with the over-rotation numbers and
over-rotation pairs of the map $f$. Thus, the classical rotation set
$I_{F_f}$ of the map $F_f$ coincides with $I_f=[\rho_f, \frac{1}{2}]$.
Moreover, comparing the maps $F_f$ and $G_f$ we see that:

			\begin{enumerate}

				\item $F_f=G_f$ except for the collection
$\mathcal{C}$ of intervals $(n+a_f', n+M_f),$ $(n+d_1(f), n+a_f),$
$(n+a_f+1-a_f'', n+a_f+1-m_f),$ $(n+a_f+1-d_2(f), n+1),$ $n \in
\mathbb{Z},$ on each of which $G_f$ is a constant;

				\item $G_f(x) \leq F_f(x)$ $\forall x \in
\mathbb{R}$;

				\item $G_f$ is continuous and non-strictly
monotonically increasing.

			\end{enumerate}

By \cite{rt86}, for every $z\in \mathbb{R}$ ,  $\displaystyle
\lim_{n\to \infty} \frac{G_f^{n}(z)}{n}=\rho'_f$ exists and is independent of
$z$; $\rho'_f$ is the \emph{classical rotation number} of the
map $G_f$. Since $G_f$ is of degree one, it induces a degree one map
$\tau_f$ of the unit circle $\uc$. Since the map $G_f$ is non-strictly
monotonically increasing, then $\tau_f$ preserves the cyclic orientation
in the non-strict sense; equivalently, one can say that $\tau_f$ is monotone
and locally non-strictly increasing (we consider counterclockwise direction
on the circle as positive). By \cite{alm00}, there exists a point $z\in
[0,1)$ whose orbit is disjoint from the union of open intervals from
$\mathcal{C}$ (thus, the $G_f$-orbit of $z$ is the same as the $F$-orbit
of $z$) and has the following properties:

\begin{enumerate}

\item if $\rho'_f$ is rational then $\pi(z)$ is periodic and in terms of circular order
$\tau_f$ acts on the orbit of $z$ as the rotation by the angle $\rho'_f$;

\item if $\rho'_f$ is irrational, then $\pi(z)\in \omega_{\tau_f}(z)$ where
$\omega_{\tau_f}(z)$ is a \emph{minimal} set (i.e. $\tau_f$-orbits of all
points of $\omega_{\tau_f}(z)$ are dense in $\omega_{\tau_f}(z)$) such that
collapsing arcs of $\uc$ complementary to $\omega_{\tau_f}(z)$ we can
semi-conjugate $\tau_f$ to the (irrational) rotation of $\uc$ by the angle
$\tau_f$.

\end{enumerate}

In the end this construction yields a (semi-)conjugation of the
original map $f$ on the limit set $\omega_f(z)$ and the rotation by
$\tau_f$ on a special set, say, $A_f$ so that (a) if $\rho'_f$ is rational
then $z$ is $f$-periodic, $A_f$ is a periodic orbit, and we deal with
conjugation, while (b) if $\rho'_f$ is irrational then $\omega_f(z)$ is a
Cantor set, $A_f=\uc$, and we deal with semi-conjugacy which is at most
two-to-one. In either case the (semi-)conjugacy acts as follows (in our
explanation we assume that the circle is normalized so that its length
is $1$): the points of $\omega_f(z)$ that belong to $[0, a_f]$ are put
on the arc $[0, a_f]$ of the circle maintaining the same order while
the points of $\omega_f(z)$ that belong to $[a_f, 1]$ are put on the
circle arc $[a_f, 1]$ in the reverse order.

It follows that in either case the map $f|_{\omega_f(z)}$ has a unique
invariant measure $\mu_f$ (i.e., it is \emph{strictly ergodic}), every
point $x\in \omega_f(z)$ is admissible, and for every point $x\in
\omega_f(z)$ we have $I_{f, \chi}(x)=\{\rho'_f\}$. Moreover, in both
cases the measure $\mu_f$ can be transformed, in a canonical fashion,
to a specific invariant measure related to the circle rotation by the
angle $\rho'_f$: in the rational case the corresponding measure is just
a CO-measure concentrated on the $f$-periodic orbit of $z$ whereas in
the irrational case it corresponds, in a canonical fashion, to the
Lebesgue measure on the unit circle invariant under the irrational
rotation by the angle $\rho'_f$.

Let us now relate $\rho'_f$ and the left endpoint $\rho_f$ of the
over-rotation interval $[\rho_f, \frac12]$ of $f$. By the above and by
Theorem~\ref{t:b1} we see that $\rho'_f\in I_f=[\rho_f, \frac12]$, and
hence $\rho_f\le \rho'_f$. On the other hand, $G_f$ is monotonically
increasing and $G_f \leq F_f$ which by induction implies that
$G_f^n(X)\le F_f^n(X) \forall n$. Indeed, the base of induction is the
fact that, by construction, $G_f\le F_f$ . Assume that the desired
inequality is proven for $n$; then, for every $X\in \R$, we have

$$G_f^n(G_f(X)) \leq G_f^n(F_f(X)) \leq F_f^n(F_f(X))\, \forall n \in \mathbb{N}$$

$$\implies G_f^{n+1}(X) \leq F_f^{n+1}(X)\, \forall n \in \mathbb{N}, X \in
\mathbb{R}$$

\smallskip

\noindent which proves the desired inequality. This implies that

$$\displaystyle \lim_{n\to \infty}
\frac{G_f^{n+1}(X)}{n}=\rho'_f \leq \lim_{n\to \infty} \frac{F_f^{n+1}(X)}{n}$$


\noindent which implies that $\rho'_f\le \rho_f$. Thus,
$\rho'_f=\rho_f$.

Since $G_f$ is non-decreasing and continuous, the map $G_f:\R\to \R$
can be monotonically semi-conjugate to a monotone circle map $h_f$ of
degree one. It is well-known (see, e.g., \cite{alm00}) that there
exists a closed invariant set $A_f$ of $h_f$ on which the rotation
number $\rho_f$ is realized; moreover, $A_f$ can be chosen to avoid
flat spots of $h_f$. Then there are two possibilities depending on
whether $\rho_f$ is rational or irrational. If $\rho_f=p/q$ is rational
with $p, q$ given in lowest terms (i.e., $p, q$ are coprime) then $A_f$
can be chosen to be a \emph{periodic orbit} of period $q$. It lifts to
the $G_f$-orbit $A_f'$ of a point $x\in \R$ such $G_f^q(x)=x+p$ so that
the classical rotation pair of $x$ under the degree one map $G_f$
(equivalently, $F_f$) is $(p, q)$. If $\rho_f$ is irrational then $A_f$
can be chosen to be a Cantor set and $h_f$ is monotonically
semi-conjugate to an irrational rotation of the circle by the map
collapsing flat spots of $A_f$.

We need to find the appropriate $f$-invariant set $Z_f$ associate to
the set $A_f$ whose existence is claimed in the theorem. In general the
situation is complicated here because of the fact that the map $F_f$
associated with $f$ does not have to coincide with $G_f$ at points of
the lifting of the set $A_f$ to the real line; in other words, in
general the set $A_f$ is not easily transformed to a closed invariant
set of $f$ on which the over-rotation $\rho$ is realized. However the
specifics of the construction allow us to circumvent these
complications.

Indeed, observe that the correspondence between the maps involved in
our construction implies the existence of a continuous conjugacy
$\psi_f$ between $f$ and $h_f$ applicable outside of the closures of
flat spots of $h_f$; the map $\psi_f$ sends orbits of $h_f$ to orbits
of $f$ while keeping the same (over-)rotation numbers in both cases.
Now, if $\rho_f$ is irrational, we can choose a point $y\in A_f$ that
avoids closures of flat spots of $h_f$ altogether. It follows that the
point $\psi_f(y)$ gives rise to its $f$-limit set $Z_f$, and since the
lifting of the set $A_f$ to the real line stays away from a small
neighborhood of $a_f$ and its integer shifts (this follows from the
fact that $\rho_f<1/2$), then there is a continuous conjugacy between
$f|_{Z_f}$ and $h_f|_{A_f}$. It is then easy to see that $Z_f$ has all
the desired properties.

Suppose now that $\rho_f$ is rational. If the corresponding periodic
orbit $A_f$ of $h_f$ avoids closures of flat spots of $h_f$ we are done
by the arguments similar to the ones from the previous paragraph.
Suppose now that $A_f$ passes through an endpoint $b$ of a flat spot of
$h_f$. Choose a point $y$ very close to $b$ avoiding flat spots of
$h_f$. Then the finite segment of the $h_f$-orbit of $y$ consisting of
points $y,$ $h_f(y)$, $\dots,$ $h_f^q(y)\approx y$ is transformed by
$\psi_f$ into a finite segment $\psi_f(y),$ $f(\psi_f(y)),$ $\dots,$
$f^q(\psi_f(y))\approx \psi_f(y)$ which converges to an $f$-periodic
orbit as $y\to b$; since in this case the lifting of $A_f$ also stays
away from small neighborhoods of $a_f$ and its integer shifts, the
limiting transition is legitimate and the limit periodic orbit $Z_f$ of
$\psi_f(y),$ $f(\psi_f(y)),$ $\dots,$ $f^q(\psi_f(y))\approx \psi_f(y)$
has all the desired properties. The remaining claims of the theorem
easily follow from the above analysis and are left to the reader.

Observe that, in case of a rational $\rho_f$,
since ultimately $Z_f$ is associated with
a cycle on the circle with a map that acts as a rotation, and since the over-rotation
pair of $Z_f$ coincides with the classical rotation pair on that cycle, it follows that
the over-rotation pair of $Z_f$ is coprime. Moreover, if we apply the proven above results
to the $Z_f$-linear map $\psi$, it follows that $Z_f$ can play the role of the $Z_\psi$; in particular
this implies that the over-rotation interval $I_\psi$ of $\psi$ equals $[\rho_f, 1/2]$ as desired.
\end{proof}

Theorem~\ref{t:bimo-twist} yields a strategy in finding the
over-rotation interval $I_f=[\rho_f, \frac12]$ of $f$: it suffices to
take any point $y$ whose trajectory is contained in $Y_f$ and compute out
its over-rotational set $I_{f, \chi}$ which, by
Theorem~\ref{t:bimo-twist}, must be a singleton $\{\rho_f\}$. Moreover, this theorem also allows
one to describe all N-bimodal over-twist patterns which is done in the next section of the present paper.

\section{$N$-BIMODAL OVER-TWIST PATTERNS}

Let us apply our results to finding the $N$-bimodal over-twist pattern.
Let $f:[0,1]\to [0,1]$ be an $N$-bimodal map for whom the notation and
agreements introduced in Definition~\ref{d:nbimo} hold. Moreover, we
will also rely upon Theorem~\ref{t:bimo-twist} and use the notation
from that theorem. Also, to emphasize the dependence on $f$, like
earlier, we will continue to use subscripts while writing $M_f, m_f,
d_1(f)$ and $d_2(f)$ to avoid any sort of  ambiguity.

First of all, we need to show that the patterns of the periodic orbits
discovered in the previous section, are over-twist patterns. What we
know is that, according to Theorem \ref{t:bimo-twist}, if $\Pi$ is such
pattern then (1) the over-rotation pair $orp(\Pi)=(p, q)$ of $\Pi$ is
coprime, and (2) if $P$ is a cycle of pattern $\Pi$ and $f$ is a
$P$-linear map, then the over-rotation interval $I_f$ of $f$ is
$[\rho(P), 1/2]$ where $\rho(P)=p/q$ is the over-rotation number of
$P$. This shows that the following theorem \cite{bb19} applies to the
above situation.

\begin{theorem}[\cite{bb19}]\label{t:endpt-ot} Let $P$ be a cycle of
	covergent pattern $\pi$ such that the $P$-linear map $f$ has the
	over-rotation interval $[\rho(P), 1/2]$ where $\rho(P)$ is the
	over-rotation number of $P$. Moreover, suppose that the over-rotation
	pair of $P$ is coprime. Then the pattern $\pi$ is over-twist.
\end{theorem}

Indeed, consider an N-bimodal interval map $f$ with rational $\rho_f$.
Consider the set $Z_f$ from Theorem \ref{t:bimo-twist}. Then the set
$Z_f$ from that theorem must be a periodic orbit of over-rotation
number $\rho_f$. Moreover, by Theorem \ref{t:bimo-twist} the
over-rotation pair of $Z_f$ is coprime and the over-twist interval of
the $Z_f$-linear map is $[\rho_f, 1/2]$. By Theorem \ref{t:endpt-ot}
it follows that the pattern of $Z_f$ is an over-twist pattern.
This completes the proof of the following corollary.

\begin{cor}\label{c:ot} Let $f$ be an $N$-bimodal map such that $\rho_f$
	is rational. Let $Z_f$ be a set defined in Theorem \ref{t:bimo-twist}.
	Then $Z_f$ is a cycle of over-twist pattern.
\end{cor}

We want to remark here, that without the assumption that the
over-rotation pair of $P$ is coprime Theorem \ref{t:endpt-ot} is not
true. Indeed, there exist non-coprime patterns whose over-rotation
number equals the left endpoint of the forced over-rotation interval,
and, by Theorem \ref{t:bm2} they are not over-twist patterns (by
Theorem \ref{t:bm2}, over-twist patterns must have coprime over-rotation pairs). In the
trivial cases these are patterns that have a block structure over
over-twist patterns. However there are similar patterns that do not
have block structure over over-twists. Such patterns are called
\emph{badly ordered} \cite{bb19}; they present a surprising departure
from the previously observed phenomenon according to which the results
about over-rotation numbers on the interval and those about classical
rotation numbers for circle maps of degree one are analogous.

Indeed,
take a circle map $f$ of degree one. Suppose that $f$ has a cycle $P$
of classical rotation pair $(mp, mq)$ which does not a block structure
over a rotation by $p/q$. Then by \cite{alm98} the rotation interval of
$f$ contains $p/q$ in its interior (in fact, the results of
\cite{alm98} are stronger and more quantitative but for our purposes the above quote is
sufficient).

Let us show that \emph{all} N-bimodal over-twist patterns can be described based upon
Theorem \ref{t:bimo-twist}.

\begin{lemma}\label{l:bimotwist} Let $f$ be a $P$-linear $N$-bimodal map
	where $P$ is a periodic orbit of over-twist pattern $\pi$ with over-rotation
	pair $(p, q)$. Then $P\subset Y_f$ can be viewed as the set $Z_f$
	from Theorem \ref{t:bimo-twist}.
\end{lemma}

Recall that the set $Y_f$ for a given N-bimodal map $f$ is defined in Theorem
\ref{t:bimo-twist}. The set $Y_f$ in Case 1 and Case 2 is shown in Figure
\ref{Y:Case1} and Figure \ref{Y:Case2} respectively.

\begin{figure}[H]
	\caption{\textbf{\textit{The set $Y_f$, shown in dotted line, for the N-bimodal map $f$ in Case 1 }}}
	\centering
	\includegraphics[width=0.55\textwidth]{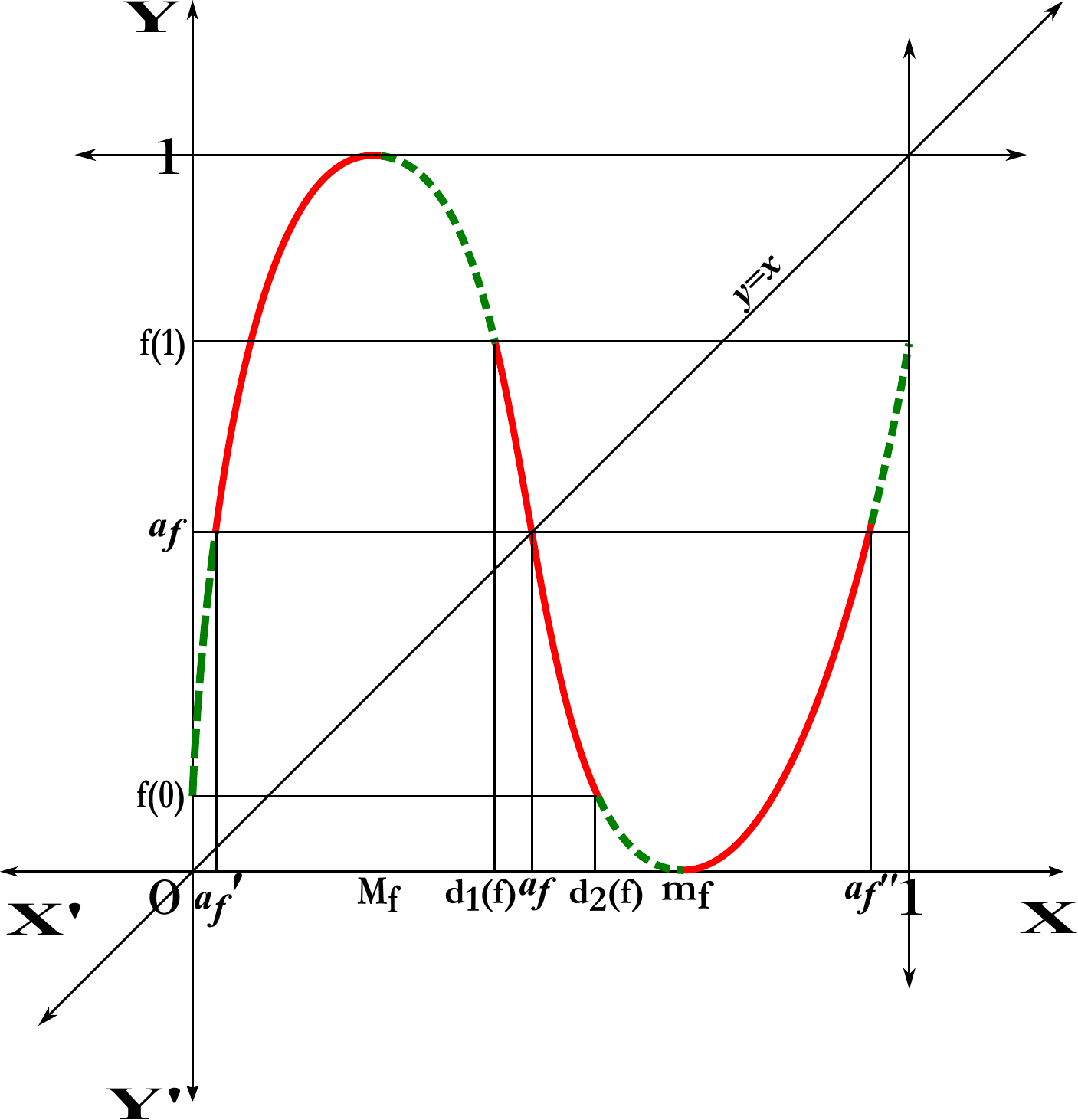}
	\label{Y:Case1}
\end{figure}

\begin{figure}[H]
	\caption{\textbf{\textit{The set $Y_f$, shown in dotted line, for the N-bimodal map $f$ in Case 2}}}
	\centering
	\includegraphics[width=0.55\textwidth]{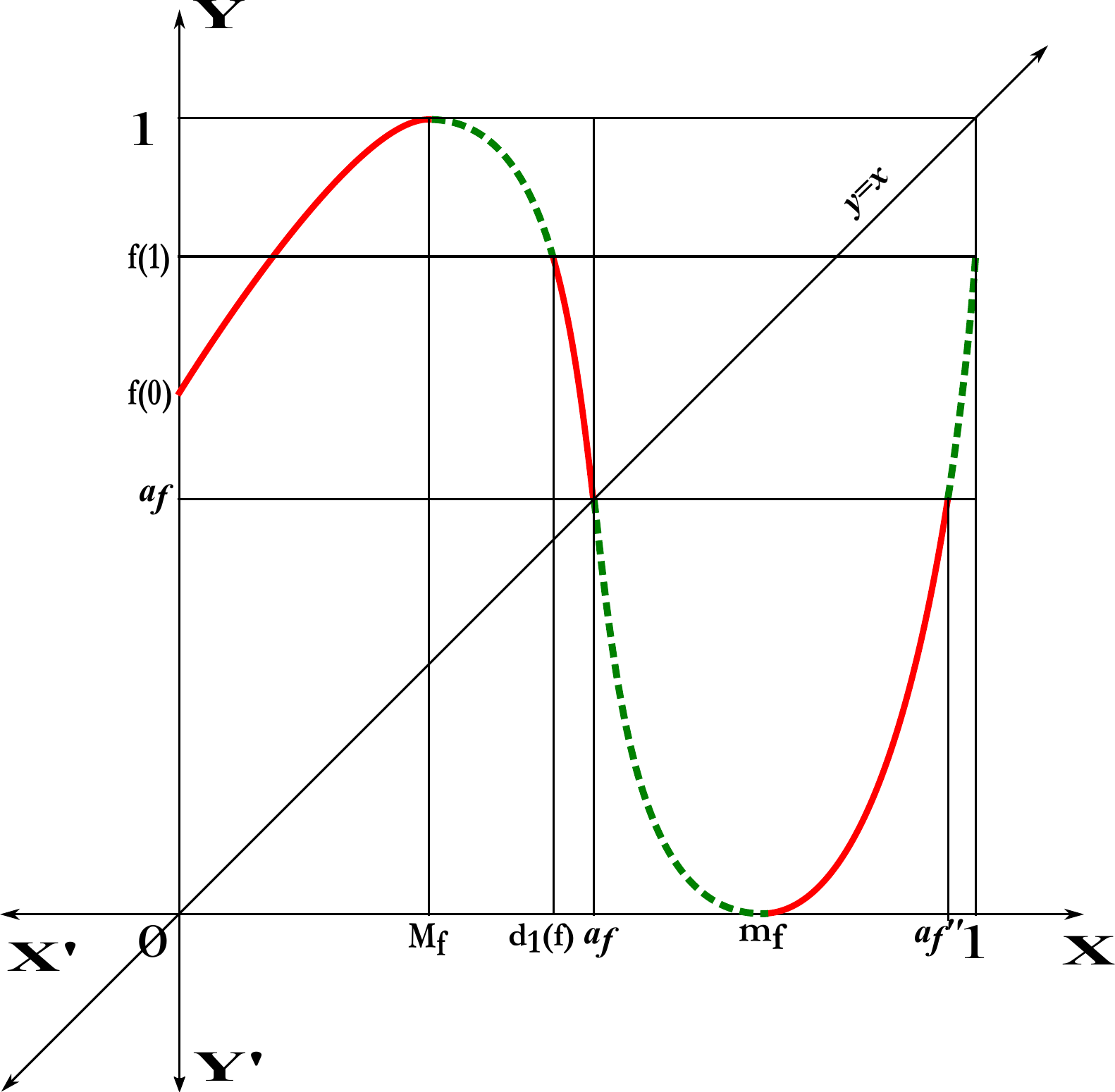}
	\label{Y:Case2}
\end{figure}

\begin{proof} By Theorem~\ref{t:bm2}, $p$ and $q$ must be coprime. Moreover,
	by Theorem~\ref{t:bm2} and by definition $I_f=[\frac{p}{q}, \frac12]$.
	Let us show that then the set $P$ is contained in $Y_f=[0, a_f']\cup
	[M_f, d_1(f)]\cup [d_2(f), m_f)]\cup [a''_f, 1]$ (if $f(0)<a_f$) or in
	$[M_f, d_1(f)]\cup [a_f, m_f]\cup [a_f'', 1]$ (if $f(0)>a_f$). It is
	sufficient to consider only the case $f(0)<a_f$ as the other one is
	similar. Suppose that the containment claimed above fails. By Theorem
	\ref{t:bimo-twist} it follows that the set $Z_f$ whose existence and
	properties are described in Theorem \ref{t:bimo-twist} is contained in
	$Y_f$ and is disjoint from $P$. Thus, $\pi$ forces a pattern $\gamma$
	of $Z_f$. However both $\pi$ and $\gamma$ have the same over-rotation
	pair $(p, q)$, which is impossible because by the assumption $\pi$ is
	an over-twist pattern.
\end{proof}

To describe all N-bimodal over-twist patterns we consider two cases.

First, assume that $f(0)\le a_f $. Set $K_1(f)=[0, a_f'],$
$K_2(f) = [M_f,$ $d_1(f)], $ $K_3(f) = [d_2(f), m_f]$, $K_4(f)=[a_f'', 1]$.
By Theorem~\ref{t:bimo-twist}, $Z_f \subset Y_f =K_1(f) \cup K_2(f)
\cup K_2(f) \cup K_3(f) \cup K_4(f)  $ which is what we will rely upon
giving an explicit description of $N$-bimodal over-twists
of over-rotation number $\frac{p}{q}$. By definition, there must be $p$
points of $P$ in the interval $K_3(f) =[d_2(f) , m_f]$ and $p$ points
of $P$ in the interval $K_2(f) = [M_f, d_1(f) ]$. Indeed, $[d_2(f),
m_f]$ is the only component of $Y_f$ to the right of $a_f$ whose points
map to the left of $a_f$ and hence contribute to the over-rotation
number. Since the number of points mapped from
the left of $a_f$ to the right of $a_f$ has to be the same, there
must be $p$ points of $P$ in the interval $K_2(f) =[M_f, d_1(f) ]$.
The remaining $q-2p$ points are contained in the intervals $K_1(f)
$ and $K_4(f)$. If there are $r$ points of $P$ in the interval $[0,
a_f']$ , then there would be $s=q-2p-r$ points in the interval
$[a_f'',1]$. This defines the number of points in the intervals
$K_1(f)$, $K_2(f),$ $K_3(f),$ and $K_4(f)$. Clearly, $r\ge 0$ and $s\ge 0$, i.e.
$0\le r\le q-2p$.

If $r=0$ or $s=q-2p-r=0$ (i.e., $r=q-2p$) , our over-twist pattern reduces to a unimodal
over-twist pattern described in \cite{BS}. First recall that the unique
unimodal over-twist pattern of over-rotation number $\frac{p}{q}$ is
denoted by $\gamma_{\frac{p}{q}}$ and its action on the $q$ points
$x_1, x_2, \dots, x_q$ of a periodic orbit $P$ which exhibits this pattern
is as follows: the first $q-2p$ points of the orbit from the left are
shifted to the right by $p$ points, the next $p$ points are flipped
(that is, the orientation is reversed, but the points which are
adjacent remains adjacent) all the way to the right. Finally, the last
$p$ points of the orbit on the right are flipped all the way to the
left. Thus $\gamma_{\frac{p}{q}}$ can be described by the permutation
$\pi_{\frac{p}{q}}$ defined as follows:

	\begin{equation}
	\pi_{\frac{p}{q}}(j)=
	\begin{cases}
	
	j+p  &   \text{if $1 \leq j \leq q-2p $}\\
	2q-2p+1-j  & \text{if $q-2p+1 \leq j \leq q-p$}\\
	q+1-j & \text{if $q-p+1\leq j \leq q$ }\\
	
	\end{cases}
	\end{equation} 		

\noindent The unimodal over-twist
pattern $\gamma_{\frac{2}{7}}$ is shown in Figure \ref{f:bimodal-1}.

\begin{figure}[H]
\caption{\textbf{\textit{The Unimodal over-twist pattern $\gamma_{\frac{2}{7}}$ }}}
\centering
\includegraphics[width=0.7\textwidth]{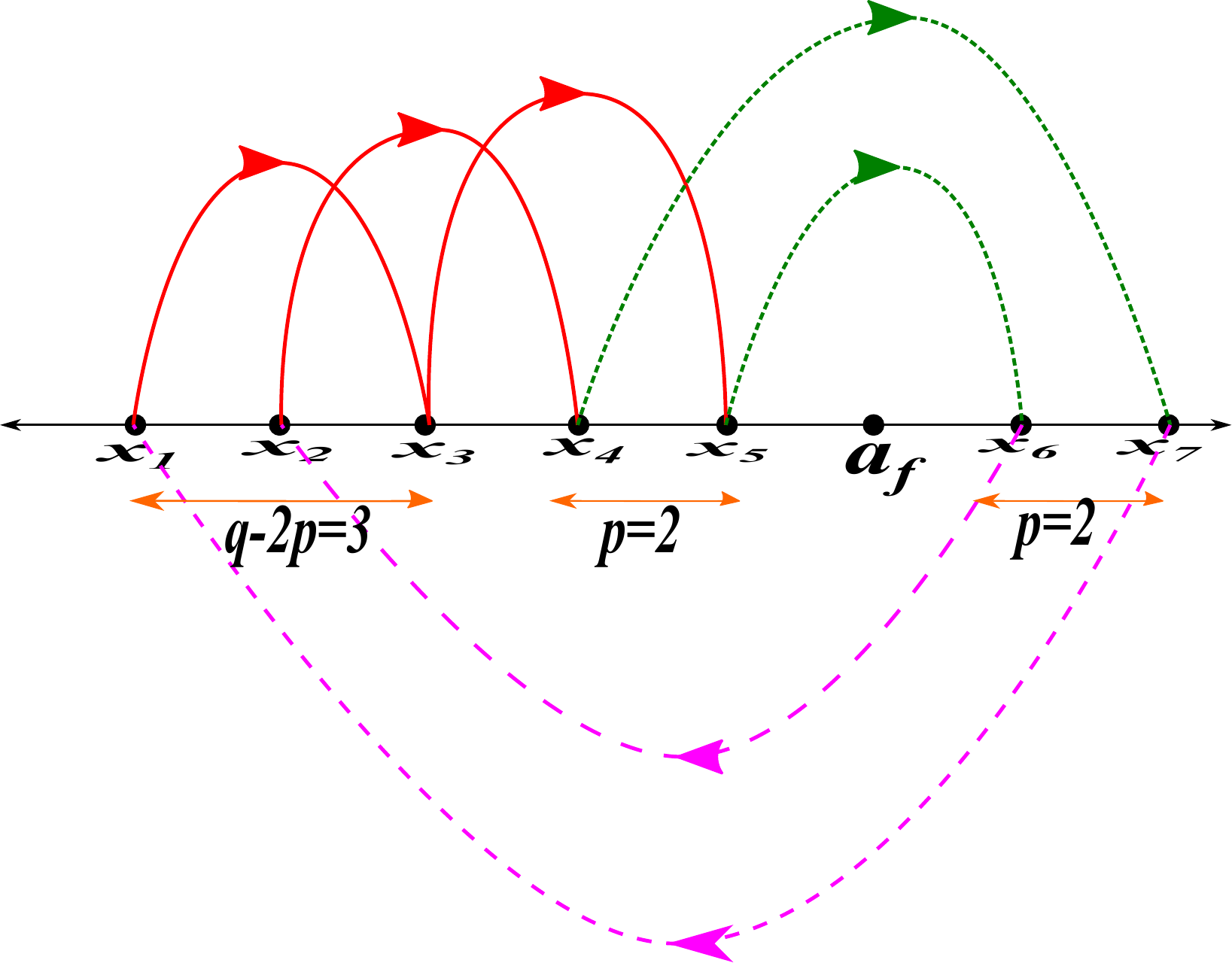}
\label{f:bimodal-1}
\end{figure}

To study over-twist patterns which are strictly bimodal, we set the
restriction $r\ge 1$ and $s\ge 1$. Then, $r \in \{1, 2, \dots,
q-2p-1\}$. Clearly, for each fixed value of $r$ from the set $\{1, 2,
\dots, q-2p-1\}$ we get a distinct bimodal over-twist pattern of
over-rotation number $\frac{p}{q}$. Thus, for the over-rotation number
$\frac{p}{q}$, there are $q-2p-1$ possible distinct bimodal
over-twist patterns each of which can be characterized by three
parameters $r, p, q$. We will denote each such patterns by
$\Gamma_{r,\frac{p}{q}}$. Let the permutation corresponding to the
over-twist pattern be denoted by $\Pi_{r, p, q}$. It follows that
$\Pi_{r, p, q}$ should be described as follows:
		
\begin{equation}
\Pi_{r,p,q}=
\begin{cases}
	
j+p  &   \text{if $1 \leq j \leq r $}\\
q-j+r+1  & \text{if $r+1 \leq j \leq r+p$}\\
2p-j+r+1 & \text{if $r+p+1 \leq j \leq r+2p$ }\\
j-p & \text {if $r+2p+1 \leq j \leq q$}\\
\end{cases}
\end{equation}

In other words, this is what the pattern $\Gamma_{r,\frac{p}{q}}$ does
with the $q$ points $x_1, x_2, \dots, x_q$ of the periodic orbit. The
first $r$ points $x_1, x_2, \dots, x_r$ from the left of the orbit are
shifted to the right by $p$ points. The next $p$ points $x_{r+1},
x_{r+2}, \dots, x_{r+p}$ map forward onto the last (the rightmost) $p$
points of the orbit with a flip (i.e., with orientation reversed) but
without any expansion so that $f(x_{r+1})) = x_q, ... f(x_{r+p}) =
x_{q-p-1}$. The images of the next $p$ points $x_{r+p+1},...x_{r+2p}$
are just the first (the leftmost) $p$ points of the orbit with a flip,
so that $f(x_{r+p+1})= x_p, \dots, f(x_{r+2p})=x_1$. Finally, the
images of the last (the rightmost) $s=q-2p-r$ points $x_{r+2p+1},
x_{r+2p+2}, \dots, x_{q}$,are exactly the points $x_{r+p+1}, \dots,
x_{q-p}$ respectively. Observe that the unimodal case
$\pi_{\frac{p}{q}}$ from \cite{BS} described above is a particular case
of $\Pi_{r,p,q}$ with $r=0$. As an example of a bimodal permutation
$\Pi_{r,p,q}$, taking $r=3$, $p=3$ and $q=11$, we get a bimodal
over-twist pattern of period 11 given by the permutation
$\Pi_{3,\frac{3}{11}}= (1,4,11,8,2,5,10,7,3,6,9)$ depicted in
Figure \ref{f:bimodal-2}.

\begin{figure}[H]
\caption{\textbf{\textit{The Bimodal over-twist pattern $\Gamma_{3,\frac{3}{11}}$ }}}
\centering
\includegraphics[width=0.7\textwidth]{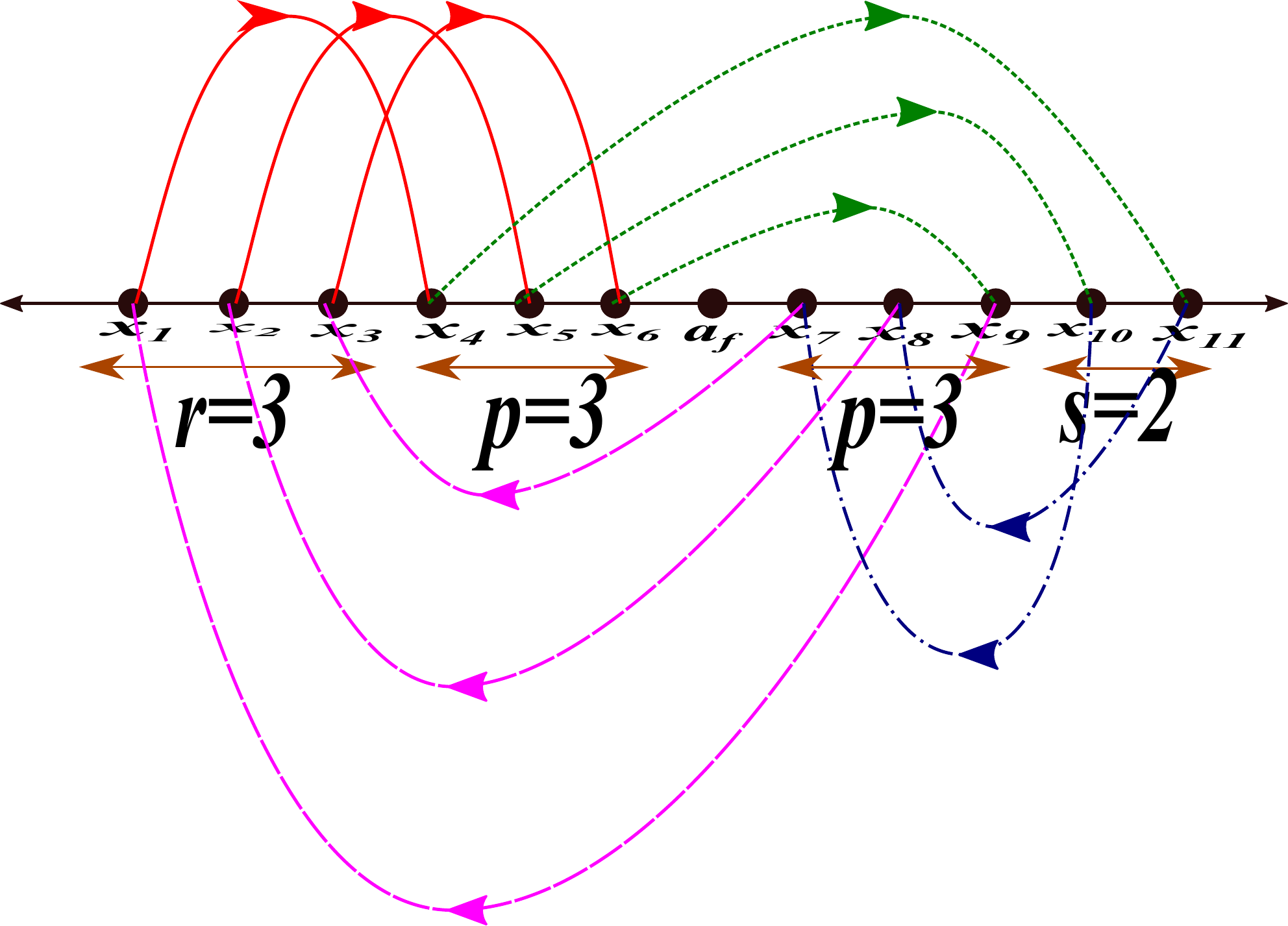}\label{f:bimodal-2}
\end{figure}

Finally, consider the case when $f(0)\ge a_f.$ Then, by
Theorem~\ref{t:bimo-twist}, the set $Z_f=P$ is a periodic orbit contained
in $Y_f=[M_f, d_1(f)]\cup [a_f, m_f]\cup [a_f'', 1]$. In such a case we see that
there must be $p$ points in each of the intervals $[M_f, d_1(f)]$ and
$[a_f,m_f]$ and $q-2p$ points in the interval $[a_f'',1]$. In such a case,
the corresponding pattern is the unimodal over-twist pattern of
over-rotation number $\frac{p}{q}$; in fact, the corresponding
permutation is the flip of the permutation $\pi_{\frac{p}{q}}$. Observe, that
according to our analysis overall there are $q-2p-1$ N-bimodal oriented over-twist patterns.


\section{WELL-BEHAVED CONTINUOUS MAPS}

In this section we extend the above results onto a wider class of
continuous interval maps which we call \emph{well behaved}. To avoid
unnecessary complications, for the sake of brevity, and to focus upon
the most interesting and broadly studied class of maps we will assume
that the maps in question are (strictly) piecewise-monotone;
however, this is not crucial and similar arguments can be applied in
the general continuous case.

\begin{definition}
Let $f:[0,1]\to [0,1]$ be a continuous map with a unique fixed point
$a_f \in (0, 1)$ (clearly, then $x<f(x)$ for any $x\in [0, a_f)$ and
$f(x)<x$ for any $x\in (a_f, 1]$; in particular, $\min_{x\in [0, a_f]}
f(x)>0$ and $\max_{x\in [a_f , 1]}f(x)<1$). Without loss of generality we
may assume that $\min_{x\in[0, 1]} f(x)=0$ and $\max_{x\in [0,
1]}f(x)=1$. Let $M_f=\max\{x: f(x)=1\}$ and $m_f=\min\{x: f(x)=0\}$
(evidently, $0<M_f<a_f$ and $a_f<m_f<1$). If for all $x\in [M_f, a_f], f(x)>a_f$ and
for all $x\in [a_f, m_f], f(x)<a_f$ we will call $f$ \emph{well behaved}. Let
$\mathcal{W}$ be the family of all well behaved maps.
\end{definition}

Figure \ref{f:wellb-1} gives an example of a well behaved map.

\begin{figure}[H]
\caption{\textbf{\textit{A well-behaved continous map $f$}}}
\centering
\includegraphics[width=0.85\textwidth]{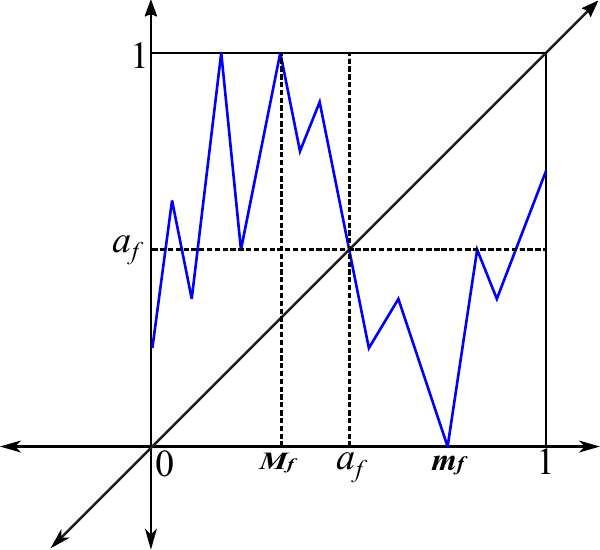}
\label{f:wellb-1}
\end{figure}

In the preceding sections of the paper we
constructed, for a given $N$-bimodal map $f$, the canonical
discontinuous lifting of $f$. It turns out that this construction can
be extended onto well behaved maps. Indeed, let $f\in \mathcal{W}$.
Like in the bimodal case, let us consider the discontinuous conjugacy
$\sigma_f:[0,1]\to [0,1]$ defined by

\begin{equation}
\sigma_f(x)=
\begin{cases}

x & \text{if $0\le x\le a_f$}\\
a_f+1-x  &   \text{if $a_f\le x\le 1$}\\
\end{cases}
\end{equation}

\noindent which conjugates $f $ to a map $g_f:[0,1]\to [0,1]$ so that $g_f=
\sigma_f\circ g_f\circ \sigma_f^{-1}$. As before, we ignore the fact that our
maps are going to be multivalued at $a_f$ and its preimages as it does
not impact the over-rotation interval of the map that depends only upon
the over-rotation numbers of periodic \emph{non-fixed} points.
As before, by $\si_f'$ we mean the map $\si_f$ restricted upon
$[a_f, 1]$; moreover, if we flip points of the plane in the vertical
direction with respect to the line $y=\frac{1+a_f}{2}$ we shall say that
we apply \emph{vertical $\si_f'$}, and if we flip points of the plane in
the horizontal direction with respect to the line $x=\frac{1+a_f}{2}$ we
shall say that we apply \emph{horizontal $\si_f'$}.

Let the preimages of $a_f$ in the interval $ [0, a_f]$ be denoted
sequentially by $a_f^{i}, i=1, 2, \dots, k$ for some $k \in \mathbb{N}$
such that $a_f^{1}<a_f^{2}<\dots <a_f^{k}=a_f$. On each interval
$(a_f^{i},a_f^{i+1})$ either $f(x)>a_f$ or  $f(x)<a_f$, and by our assumption $f(x)>a_f$ on $(a_f^{k-1}, a_f)$.

Each interval $[a_f^{i}, a_f^{i+1}], i=1, 2, \dots, k-1$ is of one
of two types. If $f(x)\le a_f$ on $[a_f^{i}, a_f^{i+1}]$, then $g_f(x)=f(x)$
and the graph of $g_f$ is same as the graph of $f$; if $f(x) \ge a_f$
on $[a_f^{i},a_f^{i+1}]$, then $g_f(x)=a_f+1-f(x)$ and the graph of $g_f$
can be obtained from the graph of $f$ by applying the vertical $\si_f'$. In this way, we can construct the graph of the map $g_f$ in the interval $[0, a_f]$ (see Figure \ref{f:wellb-2}).

\begin{figure}[H]
	\caption{\textbf{\textit{Construction of the map $g_f$ for the well behaved continuous map $f$ }}}
	\centering
	\includegraphics[width=0.9\textwidth]{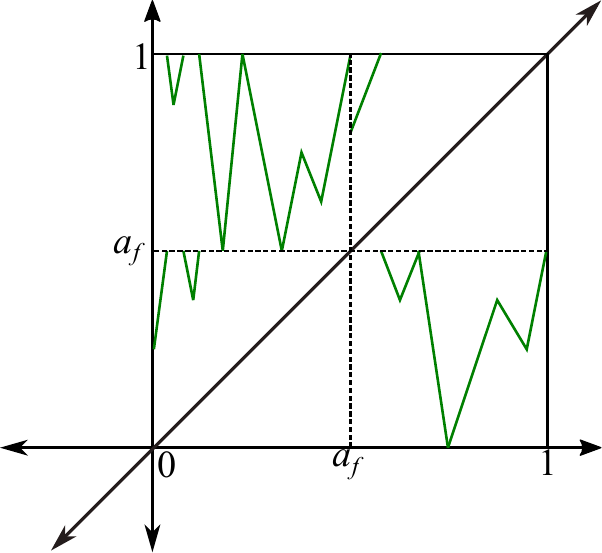}
\label{f:wellb-2}
\end{figure}

Now we construct the graph of the map $g_f$ on the interval $[a_f, 1]$.

\begin{enumerate}

\item Apply the horizontal $\si_f'$ to the entire graph of $f$
    on the interval $[a_f,1]$. Define $h_f:[a_f,1]\to [0,1]$ by $h_f(x)=
    f(a_f+1-x).$

\item Let the preimages of $a_f$ under the function $h_f$ in the
    interval $[a_f, 1)$ be denoted sequentially by
    $b_f^{1},b_f^{2},b_f^{3},....b_f^{l} $ where $b_f^{1}=a$
    $b_f^{i}<b_f^{i+1} \forall i$. In each of the intervals,
    $[b_f^{i}, b_f^{i+1}]$, the function $h_f(x)-a_f$ will have the
    same sign. By our assumption in the interval
    $(a_f, b_f^2)=(b_f^1, b_f^2),$ $h_f(x)<a_f.$

\item If for some $i\in \{1,2...k \}$ in the interval
    $[b_f^{i},b_f^{i+1}]$ we have $h_f(x)\le a_f$, then in that
    interval,  $g_f(x)=h_f(x)=f(a_f+1-x)$. On the other hand, if in the
    interval $[b_f^{i},b_f^{i+1}]$ we have $h_f(x)\ge a_f$, then in
    that interval $g(x)=a+1-h(x)=a+1-f(a_f+1-x),$ that is, in that
    case the graph of $g_f$ can be obtained by applying the vertical
    $\si_f'$ to the graph of $h_f$.

\end{enumerate}

The graph of $g_f$ thus constructed will be discontinuous but will have
the same over-rotation interval as the map $f$, i.e. $I_{g_f}=I_f$. We now
define a lifting $F_f$ of degree one of the function $f$:
\begin{equation}
F_f(x)=
\begin{cases}
g_f(x)+1 & \text{if $x\in [a_f, 1]$ and $g_f(x)<a_f$}\\
g_f(x)  &   \text{ otherwise }
\end{cases}
\nonumber
\end{equation}
and then as usual if $x = k +y $ with $y \in [0,1)$, then $F_f(x) = k+
F_f(y)$. The map $F_f$ so constructed will be a degree one map of the real
line to itself, that is, an \emph{old map} (we borrow our terminology
here from \cite{mis82}). Obviously, by the construction the sets of
classical rotation numbers and pairs of $F_f$ coincide with the sets of
over-rotation numbers and pairs of $g$ and hence with the sets of
over-rotation numbers and pairs of $f$. So, the classical rotation set
$I_{F_f}$ of the function $F_f$ coincides with the over-rotation interval
$I_f$ of the function $f$.

Observe that by construction all discontinuities of $F_f$ are at points
that map to $a_f$ and its integer shifts. Since the behavior of the map
at these points is irrelevant to our studies that concentrate upon
figuring out the left endpoint of the over-rotation interval as well as
the dynamics of over-twist patterns of over-rotation number not equal
to $1/2$, we see that a lot of arguments that apply in the continuous
case apply to our functions too. Notice also, that by construction
$F_f([0, 1])\subset [0, 2]$.

Next we construct the \emph{lower bound} function $G_f$ similar to
the corresponding function constructed previously for $N$-bimodal maps.
However here we follow the classic approach from
\cite{alm00}. The definition of the \emph{lower bound function}
$G_f$ is as follows: $G_f=\inf\{F_f(y): y\ge x\}$. Heuristically, one can get
the graph of $G_f$ from the graph of $F_f$ in the following manner: take
the graph of $F_f$ and start to pour water onto it \emph{from below} so
long that it starts to pour out over the ``edges''. Then, the bottom
level of water thus formed will give us the graph of the function $G_f$.
This function is clearly non-decreasing (in fact, if the original
function $F_f$ is non-decreasing then $G_f=F_f$). We want to discover
conditions on $F_f$ that would imply that $G_f$ is continuous because this
would in turn imply the existence of a point $x$ whose $G_f$-orbit of $G_f$
avoids ``flat spots'' of $G_f$ and, therefore, coincides with the
$F_f$-orbit of $x$. This would imply that $x$ has the lowest classic
rotation number in the sense of $F_F$, and that the corresponding point
$x'=\si_f(x)$ has the least possible over-rotation number in the sense of
$f$.

\begin{figure}[H]
\caption{\textbf{\textit{Construction of the maps $F_f$ and $G_f$ for the well behaved continuous map $f$  }}}
\centering
\includegraphics[width=1.2\textwidth]{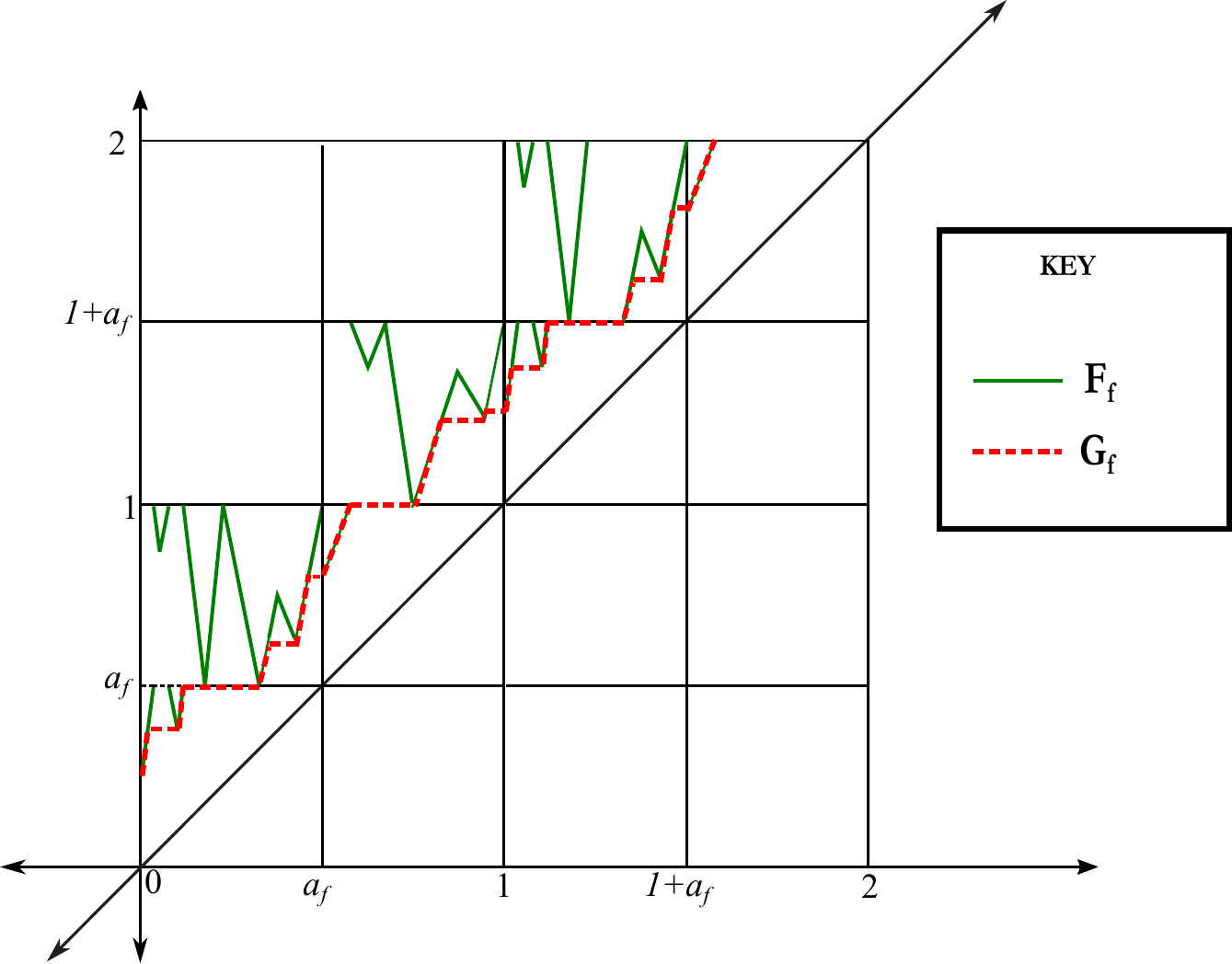}
\label{GF:well}
\end{figure}

\begin{lemma}\label{l:jumps} The range of any non-decreasing
function $\psi: \R\to \R$ is closed; it coincides with $\R$ if and only
if $\psi$ is continuous. The number of points of discontinuity of
$\psi$ is at most countable.
\end{lemma}

\begin{proof}
Clearly, $\lim_{x\to a^-} \psi(x)<\lim_{x\to a^+} \psi(x)$ for a point $a$
of discontinuity of $\psi$; depending upon the value of
$\psi$ at $a$ either one or two open intervals are going to be missing
from the range of $\psi$. Since $\psi$ is non-decreasing, such gaps in
the range are disjoint for distinct points of discontinuity of $\psi$.
Hence there are at most countably many points of discontinuity of
$\psi$. The complement to the range is always the union of those gaps,
hence the range is closed. The rest of the lemma is just as simple and
is left to the reader.
\end{proof}

We now introduce a new class of maps $\mathcal{I}$ from $\R$ to $\R$.

\begin{definition}\label{d:event-incr}
A function $T:\R\to \R$ is called \textit{eventually-increasing} if
there exists a dense set $D_T\subset \R$ such that for any $z\in D_T\,$
$\exists y\in \R $ with $T(y)=z $ and $T(x)>z $ $\forall x>y$.
\end{definition}

Thus, a map $T:\R\to \R$ is \textit{eventually-increasing} if
any horizontal line (level) from a dense family will intersect the
graph of $T $ so that there will exist a point of intersection after
which (i.e., to the right of which) the graph of $T$ will be
strictly above that horizontal line.

\begin{lemma}\label{l:eventually-increasing}
Let $T:\R \to \R $ be eventually-increasing. Then, the lower bound
function $S_T:\R\to \R$ defined by $S_T(x)=\inf\{T(y):y\ge x\}$ is
continuous.
\end{lemma}

\begin{proof}
Given $z \in D_T$ choose $y$ such that $T(y)=z$ and $T(x)>z$ for all
$x>y$. Then by definition it follows that $S_T(y)=z$. Since $D_T$ is
dense, Lemma~\ref{l:jumps} implies the claim.
\end{proof}

We will now prove the main result of this section of the paper. In
proving it we will keep intact all the notation and agreements introduced in
the paper so far. However we also need to introduce some new
concepts and notation. Suppose that $f\in \mathcal{W}$. Then $a_f$ is a
unique fixed point of $f$. We construct now a special branch of the
inverse function of $f$, i.e. a function $h_f:[0, 1]\to [0, 1]$ such
that $f\circ h_f(x)=x$ (observe that by definition $f$ is onto). The
function $h_f$, called the \emph{canonical inverse (of $f$)} is
constructed as follows. Suppose that $\al_f=\min\{f(x): x\le a_f\}$ and
$\be_f=\max\{f(x):x\ge a_f\}$. If the entire segment $[0, a_f]$ maps to
the right of $a_f$, then $\al_f=a_f$; similarly, if the segment $[a_f,
1]$ maps to the left of $a_f$, then $\be_f=a_f$. Since the case when
segments $[0, a_f]$ and $[a_f, 1]$ are flipped to the other side of
$a_f$ is trivial, we assume that at least one of them is not flipped to
the other side of $a_f$. Thus, at least one of the numbers $\al_f,
\be_f$ is not equal to $a_f$. For the sake of definiteness from now on
we assume that $\al_f<a_f$.

Now, let $z\in [\al_f, a_f]$. Then we define $h_f(z)$ as the greatest
number $y\in [0, a_f]$ such that $f(y)=z$ (in particular, $h_f(a_f)=a_f$).
Similarly, if $z\in [a_f, \be_f]$, then $h_f(z)$ is the least number $y$
such that $f(y)=z$. If now $z\notin [\al_f, \be_f]$ then we define
$h_f(z)$ as the closest to $a_f$ number $y$ such that $f(y)=z$. This
completely defines the function $h_f$. A useful exercise for the reader
here is to consider $N$-bimodal functions and describe their canonical
inverses.

We can directly describe the set $\overline{h_f([0, 1])}$ as follows. Set

\begin{multline*}
L_1(f) = \{x\in [0, a_f]: f(x)\in [\al_f, a_f]\cup (\be_f, 1], \\
 x =\sup\{y\in [0, a_f]|f(y)=f(x) \} \}
\end{multline*}

\noindent

\begin{multline*}
L_2(f)=\{x\in [a_f, 1]: f(x)\in [a_f, \be_f]\cup [0, \al_f), \\
 x=\inf\{y\in [a_f, 1]|f(y)=f(x) \} \};
\end{multline*}

\noindent then it is easy to see that $\overline{h_f([0, 1])}=L_1(f)\cup
L_2(f)=Y_f.$

Using the introduced notation we now prove our main result.

\begin{theorem}\label{t:well-behaved}
Let $f:[0, 1]\to [0, 1], f\in \mathcal{W}$ be a well behaved map. Then
the lower bound function $G_f:\R\to \R $ is continuous and increasing.
Moreover, the classical rotation number of the map $G_f$ equals the
left endpoint $\rho_f$ of the over-rotation interval $[\rho_f,
\frac12]$ of the map $f$. There exists a minimal $f$-invariant set
$Z_f$ such that for every point $y\in Z_f$ we have $I_{f,
\chi}(y)=\rho_f$, and there are two possibilities:

\begin{enumerate}

\item $\rho_f$ is rational, $Z_f$ is a periodic orbit, and
    $f|_{Z_f}$ is canonically conjugate to the rotation by $\rho_f$
    restricted on one of its cycles;

\item $\rho_f$ is irrational, $Z_f$ is a Cantor set, and $f|_{Z_f}$
    is canonically semi-conjugate by a map which is at most
    two-to-one to the circle rotation by $\rho_f$.

\end{enumerate}

\noindent Furthermore, $Z_f\subset Y_f$ and for every point $y$ whose trajectory is contained
in $Y_f$ we have  $I_{f, \chi}(y)=\{\rho\}$.
\end{theorem}

\begin{proof}
We first show that $F_f$ is \textit{eventually increasing}. Take a level
$y=\lambda$ where $0<\lambda<a_f$. By construction of $F_f$, all
$F_f$-preimages of $\lambda$ (associate to all points of intersection of
the line $y=\lambda$ with the graph of $F_f$) lie strictly to the left of
$a_f$. By construction, they are contained in $[a_f-1, a_f]$, and there are
two cases: (1) if $\al_f \le \lambda <a_f$ then some points like that belong
to $[0, a_f)$, and (2) if $0<\lambda<\al_f$ then all points like that
belong to $(a_f-1, 0)$. In either case though the continuity of $f$ (and
therefore the continuity of $F_f$ outside the set of preimages of $a_f$ and
its integer shifts) implies that there is the greatest point $y$ with
$F_f(y)=\lambda$. Now, take $t>y$. If $t\ge a_f$ then by construction
$F_f(t)\ge a_f >\lambda$. If $y<t<a_f $ then $F_f(t)$ cannot be less than
$\lambda$ as by construction at the right endpoint of the interval of
continuity of $F_f$ containing $t$ the function $F_f$ must reach out to
$a_f>\lambda$, hence by the Intermediate Value Theorem there must exist
preimages of $\lambda$ to the right of $y$, a contradiction.

The level $\lambda$ with $a_f<\lambda<1$ is considered similarly; the
difference with the previous case is only that now we have to rely upon
the fact that on any interval of continuity of $F_f$ between $a_f$ and $1$
the function $F_f$ has to reach out to the level $1$. This shows that the
the necessary conditions for $F_f$ to be eventually increasing are
satisfied for all values except for countable families oTf integer
shifts of $a_f$ and integers themselves. Hence $F_f$ is eventually
increasing and $G_f\le F_f$ is continuous.The graph of $F_f$ and $G_f$ for the \emph{well-behaved map} $f$ is shown in Figure \ref{GF:well}.

The remaining arguments literally repeat the arguments in the last part
of the proof of Theorem~\ref{t:bimo-twist} and are left to the reader.
\end{proof}

Call a pattern $\pi$ \emph{well behaved} if any cycle $P$ of pattern
$\pi$ gives rise to a well behaved $P$-linear map $f_P=f$.
Theorem~\ref{t:well-behaved}, together with the arguments used in the
proof of Corollary \ref{c:ot} and Lemma \ref{l:bimotwist},
implies the description of well-behaved over-twist patterns. Recall
that since $f$ is well behaved, for it there are several canonically defined
sets, such as the set $L_1(f)$, the set $L_2(f)$, and their union $Y_f$.

\begin{cor}
Let $P$ be a cycle of well behaved pattern $\pi$ and let $f_P$ be a
$\pi$-linear map. Then $\pi$ is an over-twist pattern if and only if
$P\subset Y_{f_P}$.
\end{cor}

\end{document}